 \documentclass[11pt,a4paper,twoside,leqno]{amsart}
\usepackage{amssymb , amsmath, amsthm}
\usepackage{amsmath}
\usepackage{amsbsy}
\usepackage{amssymb}
\usepackage{amsfonts}
\usepackage{amstext}
\usepackage{color}
\usepackage{graphicx}
\usepackage{subfigure}

\usepackage[active]{srcltx}

\newtheorem{theo}{Theorem}[section]
\newtheorem{prop}{Proposition}[section]
\newtheorem{coro}{Corollary}[section]
\newtheorem{lemma}{Lemma}[section]
\theoremstyle{definition}
\newtheorem{definition}{Definition}[section]

\newtheorem{remark}{Remark}[section]

\newcommand{\s}{{\mbox{\boldmath $s$}}}

\def\co{\mathbb C}
\def\r{\mathbb R}

\def\h{\mathbb H}
\def\s{\mathbb S}
\def\n{\mathbb N}

\newcommand{\R}{\mathbb R}

\newcommand{\C}{\mathbb C}

\newcommand{\di}{\mathbb D}

\newcommand{\wt}{\widetilde}
\newcommand{\wh}{\widehat}

\renewcommand{\Re}{{\rm Re}}
\renewcommand{\Im}{{\rm Im}}

\let\h=\hip

\def\s{\mathbb  S}

\def\rmd{\mathop{\rm d\kern -1pt}\nolimits}
\def\rme{\mathop{\rm e\kern -1pt}\nolimits}

\def\bel{ \medskip
 \centerline{$ \ast \hbox to 1.0cm{}\ast \hbox to 1.0cm{}\ast $}
}

\def\longerrightarrow{-\kern-5pt\longrightarrow}

\def\star{\lower 1pt\hbox{*}}
\def \nulset {
\raise 1pt\hbox{ \hskip -3pt$\not$\kern -0.2pt \raise
.7pt\hbox{${\scriptstyle\bigcirc}$}}}

\newcommand{\hd}{\mathbb{H}^2}
\newcommand{\sd}{\mathbb{S}^2}
\newcommand{\hi}[1]{\mathbb{H}^#1}

\newcommand{\ch}{\cosh}
\newcommand{\sh}{\sinh}

\newcommand{\pain}{\partial_{\infty}}

\newcommand{\ov}[1]{\overline{#1}}

\let\leq=\leqslant
\let\geq=\geqslant

\begin{document}

\title[ A Schoen theorem for minimal surfaces  in ${\mathbb H}^2\times {\mathbb
R}$]{A Schoen theorem for minimal surfaces  in ${\mathbb H}^2\times {\mathbb
R}$}

\author[L. Hauswirth, $\ $ B. Nelli, $\ $  R. Sa
Earp, $\ $ E. Toubiana  ]{Laurent Hauswirth, Barbara Nelli,
 Ricardo Sa
Earp,
Eric Toubiana}

\address{Universit\'e Paris-Est Marne-la-Vall\'ee \newline
Laboratoire d'Analyse et Math\'ematiques Appliqu\'ees \newline
Cit\'e Descartes \newline
5 Boulevard Descartes \newline
 Champs-sur-Marne \newline
 77454 Marne-la-Vall\'ee Cedex 2 \newline
 France}
\email{Laurent.Hauswirth@univ-mlv.fr}

\address{DISIM \newline
Universit\'a di L'Aquila \newline
via Vetoio - Loc. Coppito \newline
67010 (L'Aquila) \newline
 Italy}
\email{nelli@univaq.it}

 \address{Departamento de Matem\'atica \newline
  Pontif\'\i cia Universidade Cat\'olica do Rio de Janeiro\newline
Rio de Janeiro \newline
22453-900 RJ \newline
 Brazil }
\email{earp@mat.puc-rio.br}

\address{Institut de Math\'ematiques de Jussieu - Paris Rive Gauche \newline
Universit\'e Paris Diderot - Paris 7 \newline
Equipe G\'eom\'etrie et Dynamique,  UMR 7586 \newline
B\^atiment Sophie Germain \newline
Case 7012 \newline
75205 Paris Cedex 13 \newline
France}
\email{eric.toubiana@imj-prg.fr}

\thanks{
\\
The second, third and fourth authors were partially supported by CNPq
and
FAPERJ of
Brasil.
}

\date{\today}

\begin{abstract}
In this paper we prove that a complete minimal surface immersed
in $\hi2 \times \R,$ with
 finite total
curvature and two ends, each one asymptotic to a vertical geodesic plane, must
be a  horizontal catenoid. Moreover,
 we give a geometric description of
minimal
ends of finite total curvature in $\hi2 \times \R$.
 We also prove
that a minimal complete end $E$  with finite total curvature is properly
immersed and that the Gaussian curvature  of $E$ is locally bounded
in terms of the geodesic distance to its boundary.
\end{abstract}


\maketitle

\section{Introduction}

\hspace{.1cm} In the early eighties, R. Schoen \cite{Schoen}  proved a beautiful theorem about
minimal surfaces in Euclidean space.
Namely, a complete and
connected
minimal surface immersed in $\R^3$ with two embedded ends of finite total
curvature is
a catenoid.

 In his article, R. Schoen described the structure of finite total curvature
ends minimally embedded in $\R^3,$ relying on
 the results of   A. Huber \cite{Hub} and  R. Osserman \cite{Osserman}
about the Weierstrass representation of such ends.

\hspace{.1cm} At the beginning of this century, the discovery of a generalized Hopf
differential by U. Abresch and H. Rosenberg \cite{Abresch-Rosenberg}
stimulated the  study of minimal surfaces in
three-dimensional homogeneous manifolds.
Many new
embedded and complete minimal surfaces have been found in $\hi2 \times \R$. In
particular
J. Pyo \cite{P} and
F. Morabito and M. Rodriguez \cite{MR} have
constructed,
independently, a family of minimal embedded annuli with
finite total curvature. Each end of such annuli is  asymptotic to  a vertical
geodesic plane.
 Such  surface is called   a {\em horizontal catenoid}, see Figure
\ref{Fig-Hor-catenoid}.
\begin{figure}[!ht]
\centerline{\includegraphics[scale=0.5]{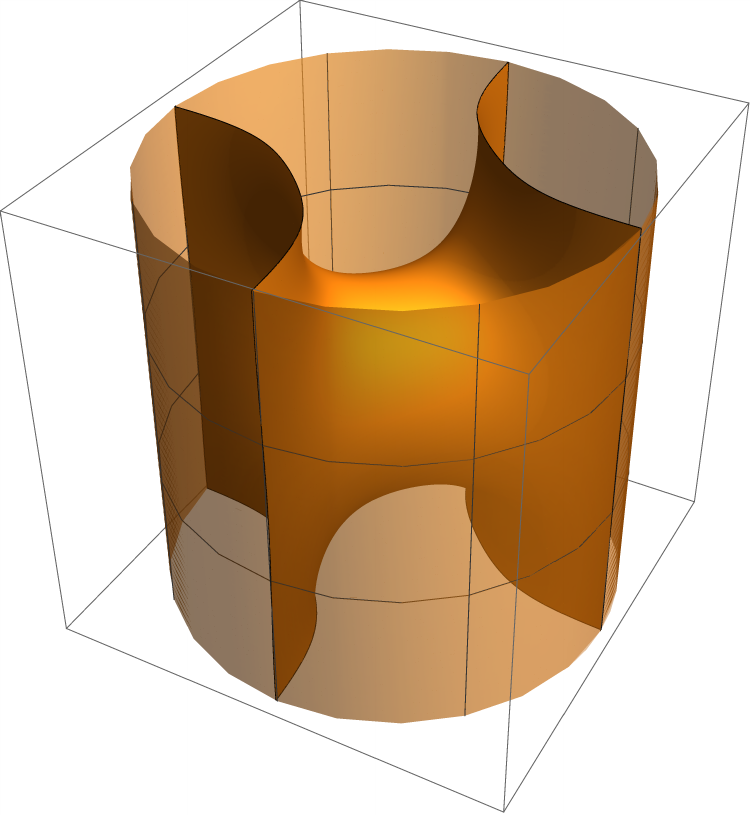}}
\caption{A horizontal catenoid in $\hi2 \times \R$
({\em courtesy of the referee})}
\label{Fig-Hor-catenoid}
\end{figure}

\

\hspace{.1cm} In this article  we prove the
following theorem.

\

\noindent {\bf Main Theorem.} {\em A complete and connected minimal surface
immersed in $\hi2
\times \R$ with {   nonzero}
 finite total
curvature and two ends, each one asymptotic to a vertical geodesic plane, is a
horizontal catenoid}.

\

Following the same spirit of Schoen's work, we describe the  full geometry
of  minimal ends  of finite total curvature in $\hi2 \times \R$ and we
give an interpretation of it in terms of closed polygonal curves
(see Definition \ref{D.Polygone} and Proposition \ref{P.Polygone}).
The study of
such ends was first developed  by the first author  and H. Rosenberg in \cite{HR}.

   We recall that in $\R^3$, there are only two kinds of embedded minimal
ends with finite total curvature: such an end is necessarily asymptotic to a
catenoid ({\em catenoidal end}) or to a plane ({\em planar end}).
It is worthwhile to notice that in $\hi2 \times \R$ there are many more
such
ends. Namely, in the Poincar\'e disk model of the hyperbolic plane, consider
the
domain $D$ with boundary the ideal polygon $\Gamma$ with vertices the $2n$
points $e^{i\frac{\pi}{k}} \in \pain \hi2$, $k=1,\dots,2n$, $n\geq 2$. Then,
P. Collin and H. Rosenberg have proved in \cite[Theorem 1]{C-R} a
Jenkins-Serrin type result:
there exists a minimal vertical graph
over $D$ taking the asymptotic values $+\infty$ and $-\infty$ alternatively on
the sides of $\Gamma$. Those examples show that there exist infinitely
many minimal embedded ends with finite total curvature in $\hi2\times \r$.

We observe that each one of those examples is properly embedded, has
finite total curvature  and one end. If $M$ is a properly embedded minimal
surface in $\hi2 \times \R$ with finite total curvature and two ends, it
is not known if each end must be asymptotic to a vertical totally geodesic
plane. For example, is it possible to connect two disjoint minimal vertical
graphs as above, with a vertical neck of a catenoid?

\

 The technical tools developped in order to prove the Main Theorem,
allow us to prove two further results. A minimal complete end with finite
total curvature is properly immersed
(Theorem \ref{T.propre}), and on such an end, say  $E$,
the Gaussian curvature is locally bounded
in terms of the geodesic distance to the boundary of $E$ (Theorem
\ref{T.estimees de courbure}).

\

The paper is organized as follows.

In Section \ref{Sec.End}, we  study  the geometry of
minimal
ends of finite total curvature. The main geometric property  is that horizontal
sections of finite total curvature ends converge towards a horizontal geodesic.
In Section \ref{Sec.Surface}, we prove the Main Theorem. In the Appendix, we
study  the geometry of curves with bounded curvature in the hyperbolic plane.

\subsection*{Acknowledgements} The second and the  fourth authors wish to thank
{\em
Departamento de
Matem\'atica da
PUC-Rio} for
the kind hospitality.
The second and third author
wish to thank {\em Laboratoire} {\em G\'eom\'etrie et Dynamique de l'Institut de
Math\'ematiques de Jussieu} for the kind hospitality.

 The authors are very grateful to the referee for the valuable observations
and for providing  the figure of a horizontal catenoid.

\section{Minimal ends with finite total curvature in $\hi2 \times
\R$}\label{Sec.End}

 \hspace{.1cm} In this section we give the geometrical  structure of a finite total
curvature end.
We rely on the  complex analysis  involved in  the  theory of minimal surfaces
 \cite{HST},  \cite{HR}, \cite{ST}
and on   the theory of harmonic maps developed  by Z. Han, L. Tan, A. Treiberg
and
 T. Wan \cite{HTTW} and Y. N. Minsky in
 \cite{Minsky}.

\hspace{.1cm} Let $M$ be a Riemann surface and let $X=(F,h) : M \rightarrow \hi2\times \R$ be
a conformal and minimal immersion. The map  $F:M\rightarrow \hi2$ is  harmonic
and $h$ is a harmonic function on $M$. Let $z$ be a local conformal
coordinate on $M$ and let $ds^2=\sigma^2 (u)\, \vert du\vert^2$ be the
hyperbolic metric on $\hi2$ in the model of the unit disk.
We set
\begin{equation*}
 Q(F):= (\sigma \circ F)^2 F_z \ov F _z dz^2 =\phi (z) dz^2,
\end{equation*}
 then
$Q(F)$  is a
quadratic holomorphic differential globally defined on  $M$, known as  the {\em
quadratic Hopf differential   associated to $F$}.

Since we consider conformal immersion we have
\begin{equation*}
\left\{
\begin{array}{l}
(\sigma \circ F)^2 |ÊF_x|^2 + h_x^2 = (\sigma \circ F)^2 |ÊF_y|^2 + h_y^2 \\
\\
(\sigma \circ F)^2 \langleÊF_x, F_y \rangle|^2 + h_x^2=0.
\end{array}
\right.
\end{equation*}
Therefore we have
$(h_z)^2(dz)^2=-Q(F)$ (see \cite[Proposition 1]{ST}). Then
$Q(F)$ has two square roots globally defined
on $M$. We denote by $\sqrt{\phi}\,dz$ the square root of $Q(F)$ so that
\begin{equation*}
 h=-2\,\Re  \int\!\! i\sqrt{\phi}\,dz =2\, \Im \int\!\! \sqrt{\phi}\,dz.
\end{equation*}

The metric induced on $M$ by the immersion $X$ is
\begin{equation*}
 ds^2 = (\sigma \circ F)^2 \Big(\rvert F_z \rvert + \rvert F_{\ov z}
\rvert\Big)^2 \rvert dz\rvert^2.
\end{equation*}



From a result by A. Huber \cite[Theorem 15]{Hub}, we deduce that a minimal end
$E$ of finite total
curvature is
parabolic, so that it can be parametrized by
$U:= \{ z\in \C \mid \ \rvert z \rvert > 1\}$.

\smallskip

Let $X=(F,h) : U \rightarrow   \hi2\times \R$ be a conformal and complete
parametrization of the end $E=X(U)$.
As it is shown in \cite{HR}, the conformal structure of the end is given by the
following
Theorem which relies  the  complex analysis  involved in  the  theory of
minimal surfaces
\cite{HST},  \cite{HR}, \cite{ST}, on the theory of harmonic maps
developed  by
Z. Han, L. Tan, A. Treiberg
and  T. Wan \cite{HTTW} and by Y. N. Minsky in
 \cite{Minsky}.

\vskip 0.5cm

{\bf Theorem \cite{HR}. }{\it  Let $X:=(F,h): M \to \hi2 \times \R$ with finite
total curvature. Then
\begin{enumerate}
\item $M$ is conformally $\bar M - \{p_1,...p_n \}$ a Riemann surface punctured
in a finite number of points.

\item $Q$ is holomorphic on $M$ and extends meromorphically to each puncture.

\item The third coordinate of the unit normal vector $n_3$ tends to zero
uniformly at each puncture.

\item The total curvature is a multiple of $2\pi$, namely
\begin{equation*}
 \int_M (-KdA)= 2 \pi (2-2g -2k-\sum_{i=1}^n m_i),
\end{equation*}
\end{enumerate}
where $m_i$ is defined in Definition  \ref{D.degre} below.
}

\medskip

This theorem contains informations on the geometrical structure of a finite
total curvature end at infinity.

\medskip

By the previous Theorem, $\phi (z)$ extends
meromorphically
to the puncture $z = \infty$. Thus we can write $\phi $ in the following
form


\begin{equation}\label{Forme generale}
 \phi (z) = \big(\sum_{k \geq 1} \frac{a_{-k}}{z^k} + P(z)\big)^2,
\end{equation}
where $P$ is a polynomial function. If we choose
$\sqrt{\phi}=\sum_{k \geq 1} \frac{a_{-k}}{z^k} + P(z)$, then
\begin{equation*}
 h=2\, \Im \int \big(\sum_{k \geq 1} \frac{a_{-k}}{z^k} + P(z)\big) dz.
\end{equation*}

\begin{definition}\label{D.degre}
 {   Let $m \geq 0$ be the degree of $P$. We will say that
{\it $E$ is an end of degree $m$ with respect to the parametrization $X$}.}
\end{definition}



Since the height function is well defined on $U$, the real part of $a_{-1}$ is
zero. Let $\beta\in\R$ such that $a_{-1}=i\beta$.

\begin{lemma}\label{L.Polynome}
 The polynomial function $P$ is not identically zero.
\end{lemma}

\begin{proof}
 Assume by contradiction  that $P\equiv 0$.  If $a_{-1}=0$ we obtain that
\begin{equation*}
 \int_{U } \rvert \phi (z) \rvert dA <\infty,
\end{equation*}
and it is shown in  \cite{HR} that the minimal end $E$ would have finite area.
From \cite{Frensel} (Theorem 3 and Remark 4) we  deduce that for any
$p\in E$ and
for any real number 
$\mu < d_E(p,\partial E)$, we have
$Area\big(B(p,\mu)\big)
\geq \pi
\mu^2$, where
$B(p,\mu)$ is the geodesic disk in $E$ centered at $p$, with radius $\mu$.
Considering a suitable diverging sequence  of points $(p_n)$ in  $E$, we
deduce
that $E$ has infinite area. This  gives a contradiction.

\

Assume now that $a_{-1}\not=0.$ Since $a_{-1}=i\beta,$
we obtain (up to an additive constant)
 \begin{equation*}
 h=2\, \Im \int \big(\sum_{k \geq 1} \frac{a_{-k}}{z^k} \big) dz
= 2\beta   \log \rvert z \rvert + o(1),
\end{equation*}
where $o(1)$ is a function depending of $z$ and $o(1) \to 0$ when
$| z| \to \infty$.

 For $R>1,$ let $A_R= \{ R \leq \rvert z \rvert \leq  R^2
\}$.
Thus, $X(A_R)$ is a~ compact and minimal annulus immersed in
$\hi2 \times \R$,   whose boundary has two connected components.
 For $R$ large enough, the vertical distance between those two boundary
components
is larger than $2 \pi,$  while  the family of the
catenoids stays in a slab of height smaller than $\pi$  \cite[Proposition
5.1]{NSWT}.
Therefore, we can compare $X(A_R)$ with the catenoids
and obtain a contradiction by the maximum principle {   since the height of
$X(A_R)$ is greater that $2\pi$}. This concludes the
proof.
\end{proof}

\hspace{.1cm} Let $E$ be an end of degree $m.$
Up to a
change of variable, we can assume that the
coefficient of the leading term of $P$ is one. Then, for suitable
complex number  $a_0,\dots, a_{m-1},$ one has

\begin{equation}\label{Eq.phi}
 P(z)= z^m +a_{m-1} z^{m-1}+ \cdots + a_0 \quad \text{and}\quad
\sqrt{\phi}=z^m (1 + o(1)).
\end{equation}
For any $R>1$, we set $U_R:=\{z\in \C \mid \  \rvert z\rvert > R\}$,
$S_R:=\{z\in \C \mid \  \rvert z\rvert = R\}=\partial U_R$ and
$E_R := X(U_R).$

\

 We set
\begin{equation*}
 W(z):=\int  \sqrt{\phi (z) }\, dz = \int \big(\sum_{k \geq 1}
\frac{a_{-k}}{z^k} + a_0 +
\cdots + z^m\big)\, dz,
\end{equation*}
so that $h(z)=2\,\Im\, W(z)$. If $\beta=0$, the function $W$ is well defined on
$U$. If $\beta \not= 0$, the function $W$ is only locally defined and has a
real period equal to $-2\pi \beta$. We denote by $\theta \in \R$ a
determination of the argument of $z\in U$, therefore
\begin{equation}\label{Eq.Ima}
\frac{1}{2}h(z)= \Im\, W(z) = \beta  \log \rvert z\rvert \,+\,
\frac{\rvert z\rvert^{m+1}}{m+1}\big( \sin (m+1)\theta + o(1)\big)
\end{equation}
and, locally
\begin{equation}\label{Eq.Reel}
\Re \,W(z) = -\beta \theta \,+\,
\frac{\rvert z\rvert^{m+1}}{m+1}\big( \cos (m+1)\theta + o(1)\big).
\end{equation}

\

\bigskip

\centerline{\sc { The  image of $W$ and   the  level sets of ${\rm Im}\, W$} }

\

\begin{definition}
(1) For any $R\geq 1$, a {\em semi-complete curve} in $U_R$ is
the image of a map
$c: [0,+\infty \mathclose[ \rightarrow U_R$ such that
$\rvert c(t)\rvert \xrightarrow[t\to \infty]{} +\infty$.

\noindent
(2)  Let $c: [0,+\infty \mathclose[ \rightarrow U_R$ be a semi-complete curve
 and let $\theta_0$ be
a real number. We say  that
 the image of $c$ has the ray $\{ r e^{i\theta_0},\ r>0\}$
as {\em asymptotic direction}, if
$\theta (t) \xrightarrow[t\to \infty]{}~ \theta_0$, where $\theta(t)$ is the
determination of the argument of $c(t)$ in
$[\theta_0-\pi,\theta_0+~\pi\mathclose[$.
\end{definition}

\

From formula (\ref{Eq.Ima}) above, by a continuity argument we deduce the
following facts.

\

\begin{lemma}\label{L.Description}
{\rm (1)}  There exists $R_0>1$ so that,
 for $k=0,\dots, 2m+~1$ and for any $R\geq R_0$, the function
$\Im\, W$ is strictly monotonous along the pairwise disjoint arcs
\begin{equation*}
A_k(R):=\big\{ z\in S_R,\   \frac{k\pi}{m+1}
-\frac{\pi}{10(m+1)}< \arg (z)
<\frac{k\pi}{m+1} +\frac{\pi}{10(m+1)}\big\}.
\end{equation*}
\noindent
{\rm (2)} For any fixed $C\in \R$ one has

\begin{itemize}
\item If $(z_n)$ is a sequence of complex
numbers such that
$\rvert z_n\rvert \to \infty$ and $\Im\,W(z_n)\equiv C$, then
$\sin \big((m+1)\arg z_n\big) \to 0$.

\item There exists $r(C) >R_0$ such that, for any $R\geq r(C)$,
there are
exactly $2m+2$ points $Re^{i\theta_k}$, $k=0,\dots,2m+1$, on the circle $S_R$
verifying $\Im\, W( Re^{i\theta_k})=C$ and
$Re^{i\theta_k} \in A_k(R) $. Moreover, we have
$\theta_k \xrightarrow[R\to \infty]{} \frac{k}{m+1}\pi$.

\item For any $R\geq r(C)$, the set
$ U_{R} \cap \{\Im\, W(z)=C\}$
is composed of $2m+2$  semi-complete curves
$H_k(C,R)$, $k=0,\dots,2m+1$. Moreover $H_k(C,R)$ has the ray
$\{re^{i\frac{k}{m+1}\pi},\ r>0\}$
as asymptotic direction.
\end{itemize}

\end{lemma}

 Let  $k=0,\dots, 2m+1.$ We take $C=0$ in Lemma \ref{L.Description} and
define
$H_k(R):= H_k(0,R),$ $R_1:=r(0)>R_0.$ Moreover  set $\alpha_k:=
\frac{k\pi}{m+1}.$ Then, we deduce the following result.

\begin{coro}\label{C.Ouverts modeles}
For any  $R\geq R_1$, the level set
$ U_{R} \cap \{\Im\, W(z)=0\}$
is composed of  $2m+2$  semi-complete curves
$H_k(R),$ $k=0,\dots,2m+1$, having the following properties
$($see Figure $\ref{F.Hk})$.
\begin{itemize}
 \item Each curve $H_k(R)$ has a unique boundary point, it
belongs to the open arc
$A_k(R)$.
\item   Each curve $H_k(R)$  has the ray $\{re^{i\frac{k\pi}{m+1}},\ r>0\}$
as asymptotic direction.
\item  Each curve $H_k(R)$ is contained in the truncated sector $\Delta_k(R)$
defined as follows
\begin{align*}
 \Delta_k(R)  &:= \big\{  \rvert z \rvert>R \ \text{and}\
 \alpha_k
-\frac{\pi}{10(m+1)}< \arg (z)
< \alpha_k  +\frac{\pi}{10(m+1)}\big\}
\end{align*}
\end{itemize}

\end{coro}

\begin{figure}[!ht]
\centerline{
\subfigure[The curves  $H_k(R)$]{
\includegraphics[scale=0.3]{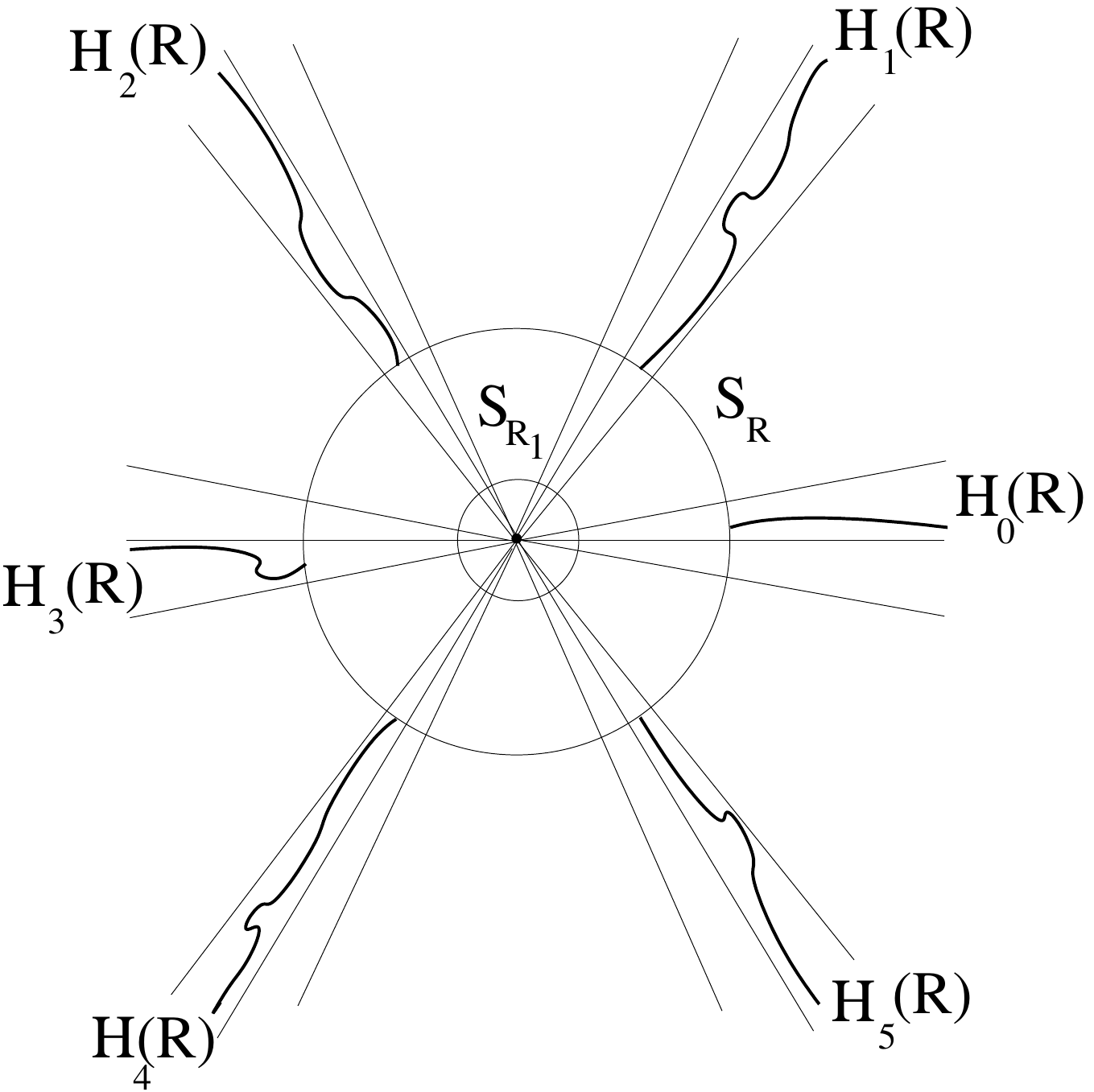}
\label{F.Hk} 
} \hskip18mm
\subfigure[The curves $L_j^+$ and  $L_j^-$]{
\includegraphics[scale=0.3]{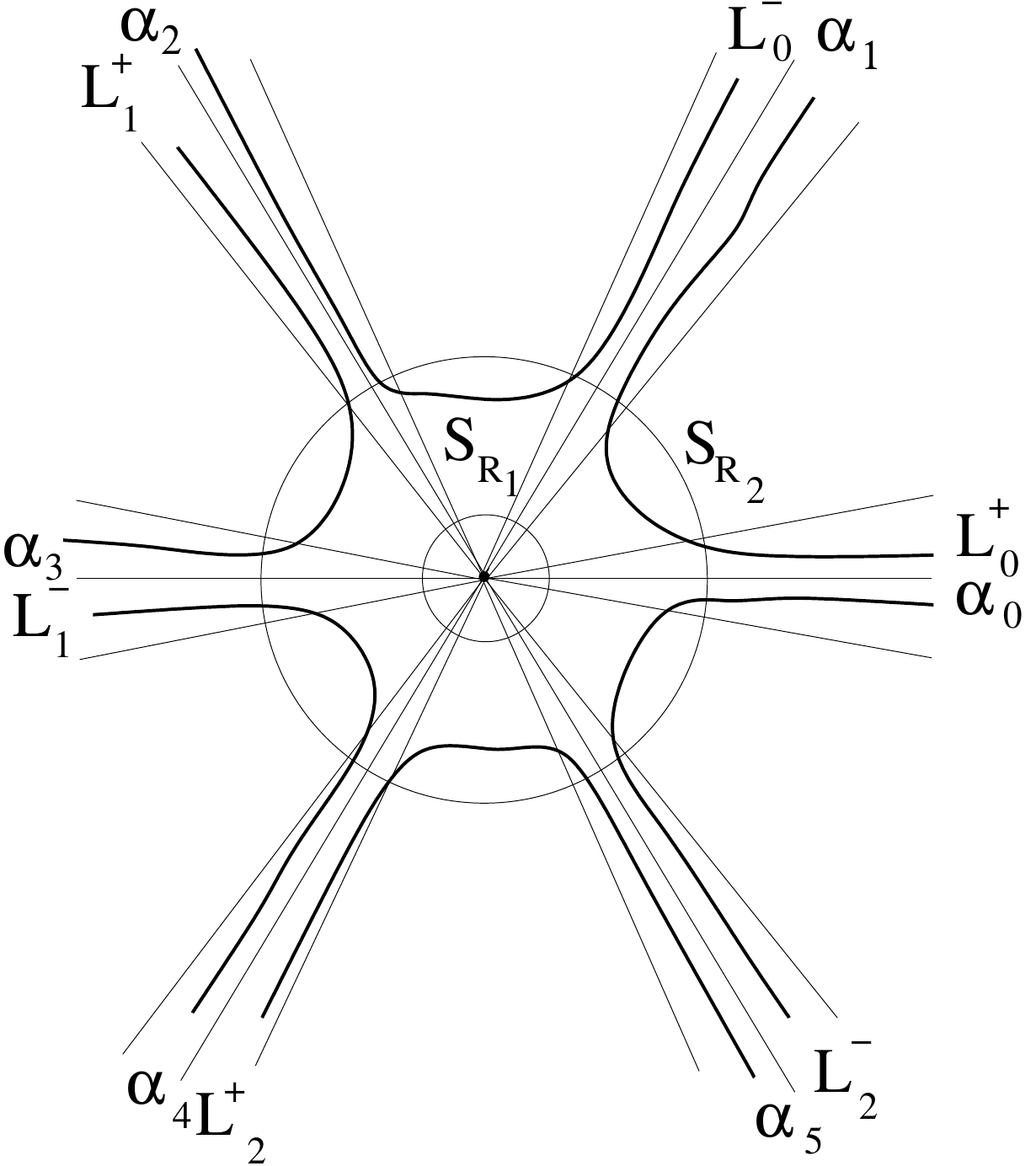}
\label{F.Lk} 
}
}

\caption{The curves $H_k(R)$, $L_j^+$ and  $L_j^-$  for $m=2$}\label{Fig1}
\end{figure}

\hspace{.1cm}  Let us state some consequences of the properties of the harmonic function
${\rm Im}\,W.$

Let $C_0 >0$ be a real number such that
$C_0 > \max \{ \rvert \Im\, W(z) \rvert,\ z\in S_{R_1}\}$.
Let $R_2$ be a real number satisfying $R_2 >r(C_0), r(-C_0), R_1$,
where $r(C_0), r(-C_0)$ and $R_1=r(0)$ are as in Lemma \ref{L.Description}. Note that the set
$ U_{R_1} \cap \{\Im\, W(z)=C_0\}$,
is composed of $m+1$ proper and complete curves without boundary
$L_0^+,\dots,L_{m}^+$ (see Figure \ref{F.Lk}).

For each $j=0,\dots,m$, the level curve
$L_j^+$ is contained in the domain of $\C$ which does not contain $0$
and which is
bounded by $H_{2j}(R_1)$,
$H_{2j+1}(R_1)$ and an arc of $S_{R_1}$ contained in the arc
$\{z\in S_{R_1} \mid \alpha_{2j}-\frac{\pi}{10(m+1)}
< \arg z <   \alpha_{2j+1}+\frac{\pi}{10(m+1)}\}$.

In the same way, the set
$ U_{R_1} \cap \{\Im\, W(z)=-C_0\}$
is composed of $m+1$ proper and complete curves without boundary
$L_0^-,\dots, L_m^-$. Each level curve $L_j^-$ is contained
in the domain of $\C$ which does not contain $0$ and which is bounded by
$H_{2j+1}(R_1)$,
$H_{2j+2}(R_1)$ and an arc of $S_{R_1}$ contained in the arc
$\{z \in S_{R_1} \mid \alpha_{2j+1}-\frac{\pi}{10(m+1)}
< \arg z <  \alpha_{2j+2}+\frac{\pi}{10(m+1)}\}$, were we set
$H_{2m+2}(R_1) := H_0(R_1)$.


 For each level curve $L_j^{\pm}$, we denote by $\mathcal{L}_j^\pm$, the
connected component of $\C\setminus L_j^{\pm}$ which does not contain the
circle $S_{R_1}$.

\begin{figure}[!ht]
\centerline{
\subfigure[$k$ even]{
\includegraphics[scale=0.3]{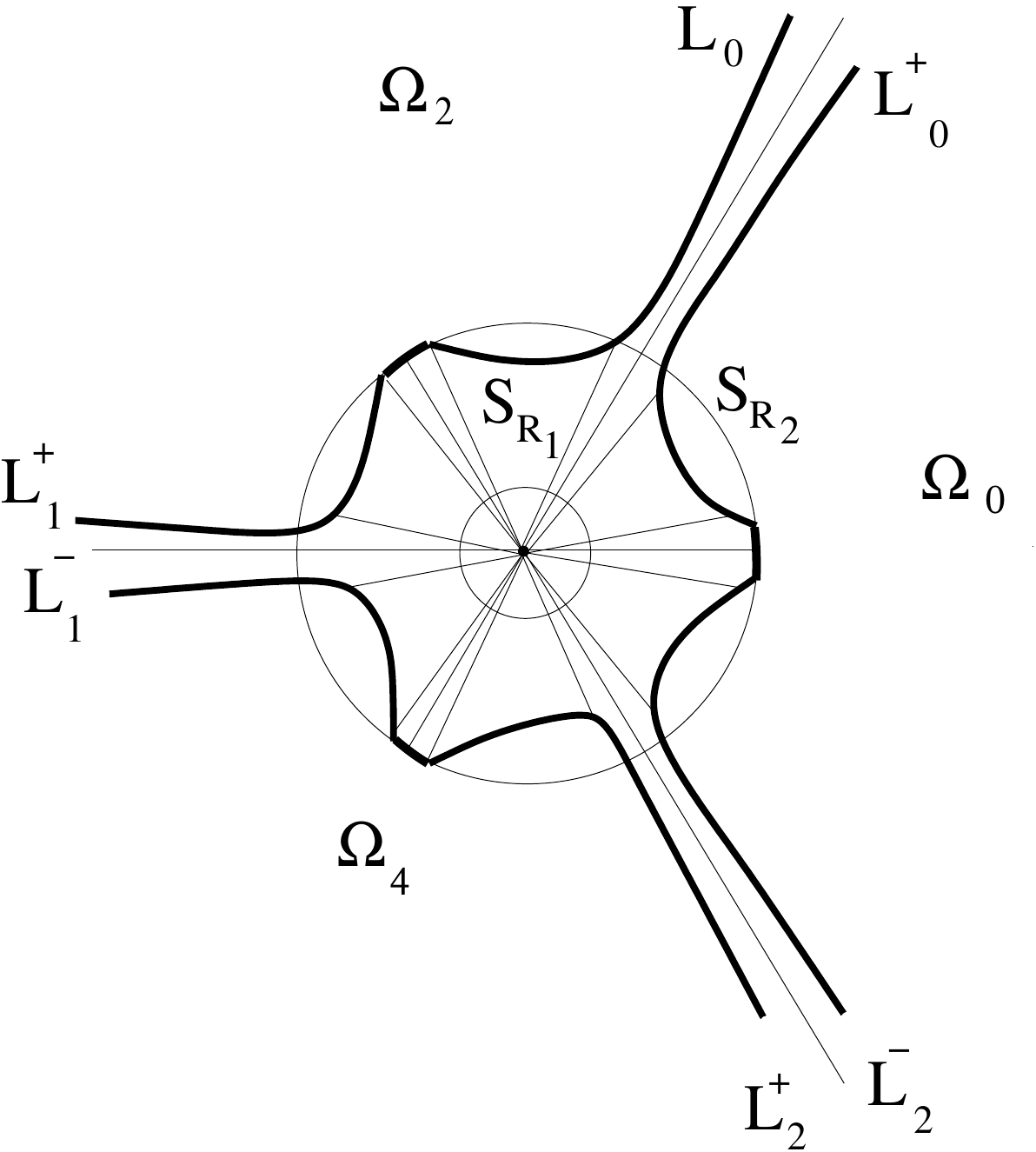}
\label{F.even} 
} \hskip30mm
\subfigure[$k$ odd]{
\includegraphics[scale=0.3]{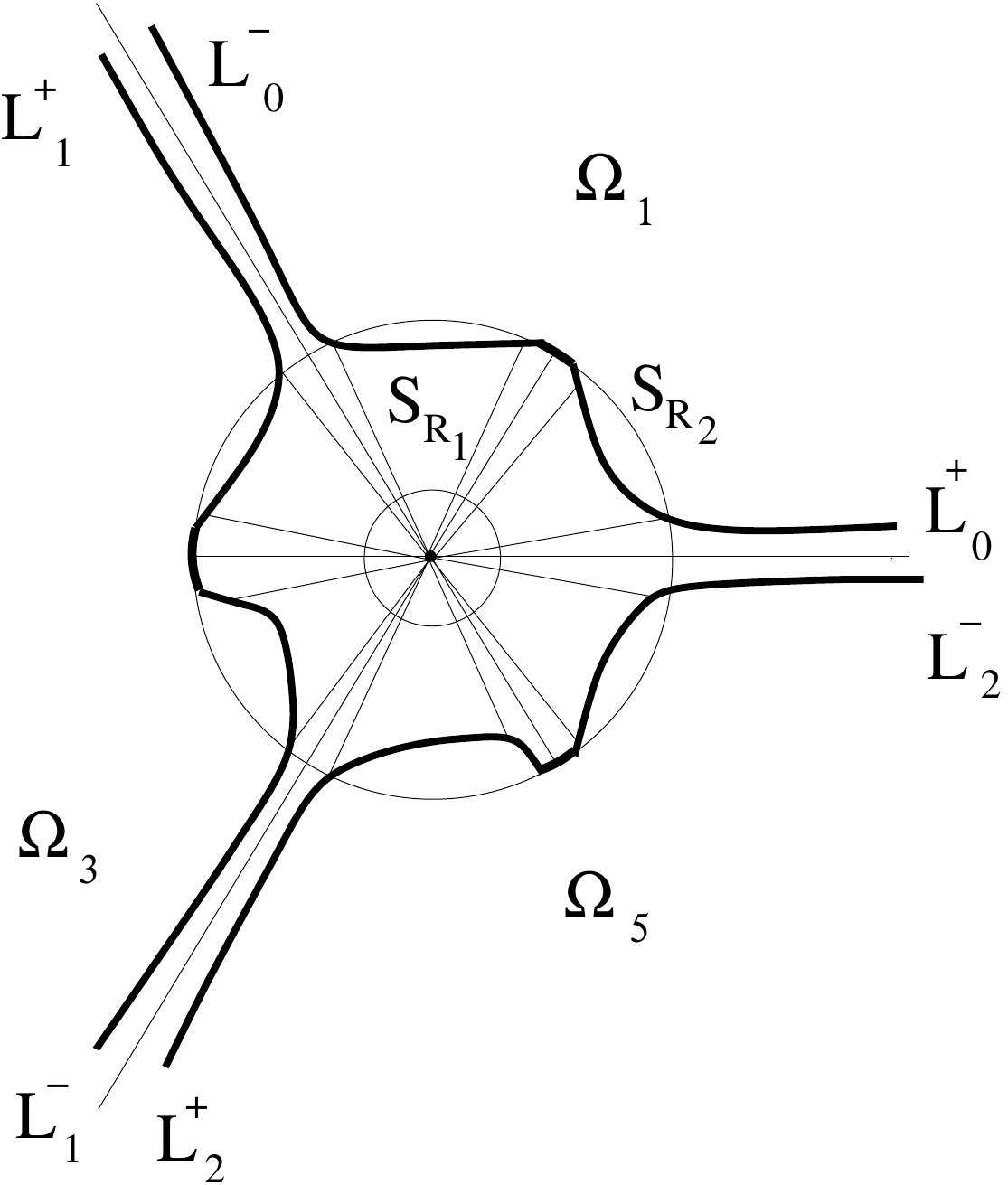}
\label{F.odd} 
}
}
\caption{The domains  $\Omega_k$ for $m=2$ }\label{Fig3}
\end{figure}

\noindent
For each $k=0,\dots,2m+1$ we define the open set $\Omega_k$ setting:
\begin{equation}\label{Eq.Open}
 \Omega_k:=
\begin{cases}
\mathcal{L}_{\frac{k}{2}-1}^- \cup \mathcal{L}_{\frac{k}{2}}^+ \cup
\Delta_k (R_2)  & \text{if}\ k\ \text{is even},\\
\mathcal{L}_{\frac{k-1}{2}}^+ \cup \mathcal{L}_{\frac{k-1}{2}}^- \cup
\Delta_k (R_2)  & \text{if}\ k\ \text{is odd},\\
\end{cases}
\end{equation}
where we set $\mathcal{L}_{-1}^-:=\mathcal{L}_m^-$ (see Figure \ref{Fig3}).

\noindent
By construction, we have that each $\Omega_k$ is a simply connected domain and,
setting $\mathcal{U}:= \cup_{k=0}^{2m+1} \Omega_k$, we have
$U_{R_2} \subset \mathcal{U} \subset U_{R_1}$.
{   Since $\Omega_k$ is simply connected, we can define a continuous
determination of the argument of $z$ in  $\Omega_k$ such that}
\begin{equation*}
 \Omega_k \subset \big\{  \rvert z \rvert>R_1 \ \text{and}\
 \alpha_{k-1}
-\frac{\pi}{10(m+1)}< \arg (z)
< \alpha_{k+1}  +\frac{\pi}{10(m+1)}\big\},
\end{equation*}
(recall that $\alpha_{-1}:=-\pi/(m+1)$ and
$\alpha_{2m+2}:= 2\pi$).

\

We summarize the above construction as follows.


\begin{lemma}\label{L.Level set}
 Let $C$ be a real number,  then the following facts hold.

 \noindent
{\rm (1)}  If $C > C_0$ then
\begin{itemize}
\item the level set  $\{\Im\, W(z)=C\} \cap \mathcal{U}$
is composed of $m+1$ proper and complete curves without boundary
$L_0(C),\dots,L_{m}(C)$ satisfying
$L_j(C)\subset \mathcal{L}_{j}^+ $ and, therefore,
$L_j(C)\subset \Omega_{2j} \cap \Omega _{2j+1}  $, $j=0,\dots, m$.

\item the level set  $\{\Im\, W(z)=-C\} \cap \mathcal{U}$
is composed of $m+1$ proper and complete curves without boundary
$L_0(-C),\dots,L_{m}(-C)$ satisfying
$L_j(-C)\subset \mathcal{L}_{j}^- $ and, therefore,\\
$L_j(-C)\subset \Omega_{2j+1} \cap \Omega _{2j+2}  $,
$j=0,\dots, m$, where $\Omega _{2m+2}:=\Omega_0$.
\end{itemize}

{\rm (2)} If $-C_0 \leq C \leq C_0$ then  the level set
$\{\Im\, W(z)=C\} \cap \mathcal{U}$
is composed of $2m+2$ proper curves
$H_k(C)\subset \Omega_k$, $k=0,\dots,2m+1$, satisfying the same properties as
the level
curves $H_k(R_2)$ in Corollary  $\ref{C.Ouverts modeles}$, with
$R=R_2$.
\end{lemma}

\

\begin{prop}\label{P.Restriction}
 For $k=0,\dots,2m+1$,
the restriction of $W$ to $\Omega_k$ is a well defined complex function,
 denoted by $W_k$. Furthermore, $W_k : \Omega_k \rightarrow \C$ is
one-to-one and defines a
conformal diffeomorphism from $\Omega_k$ onto a simply connected domain
$\wt\Omega_k := W_k ( \Omega_k)$
in the $w$ complex plane.
\end{prop}

\begin{proof}
 Since $\Omega_k$ is a simply connected domain which does not contain the
origin, the function $W$ is well defined on $\Omega_k$.

Let $z_1, z_2 \in \Omega_k$ be such that $W_k (z_1) = W_k (z_2)$. We deduce
from
Lemma \ref{L.Level set} that for any $C\in \R$, the level set
$\{\Im\, W_k (z) =C\}$
has a unique connected component in $\Omega_k$. Therefore, $z_1$ and $z_2$
belong
to the same level curve $L\subset \Omega_k$. Since
$W_k^\prime (z)=\sqrt{\phi (z)}$ and $\phi$ does not vanish on $\mathcal{U}$, we
deduce that the function $\Re\, W$ is strictly monotonous on $L$. We conclude
that $z_1=z_2$ as desired.
\end{proof}

\bigskip

 In the $w$-complex plane the domains $\wt \Omega_k, $
$k=0,\dots,2m+1$, defined in Proposition \ref{P.Restriction}, have a nice
structure, that will be crucial in the following.

\begin{coro}\label{C.Description}
Let $k$ be an even number, $k=2j$. Then, $\wt \Omega_k$ is the
complementary of a horizontal half-strip. The non horizontal component of
$\partial \wt \Omega_k$
is  a compact arc that is the image by $W_k$ of the boundary arc of
$\Omega_k$ in  $A_k(R_2)$ joining $L_j^+$ and $L_{j-1}^-$.
Thus,  $\Im\, W$ is strictly monotonous along such non horizontal component
and $\Re\, w$ is bounded from
above by
a real number
$a_k$ for any $w\in \partial  \wt \Omega_k$
$($see Figure $\ref{F.pair})$.

If $k$ is an odd number, then $\wt \Omega_k$ has a similar description, except
that on the half-strip the real part of $w$ is now bounded from below, i.e.
for some real number $b_k$ we have $\Re\, w > b_k$
for any  $w\in \partial \wt \Omega_k$ $($see Figure $\ref{F.impair})$.
\end{coro}

 We get a proof of Corollary \ref{C.Description} by  invoking Lemma
\ref{L.Level
set}.

\begin{figure}[!ht]
\centerline{
\subfigure[$k$ even]{
\includegraphics[scale=0.3]{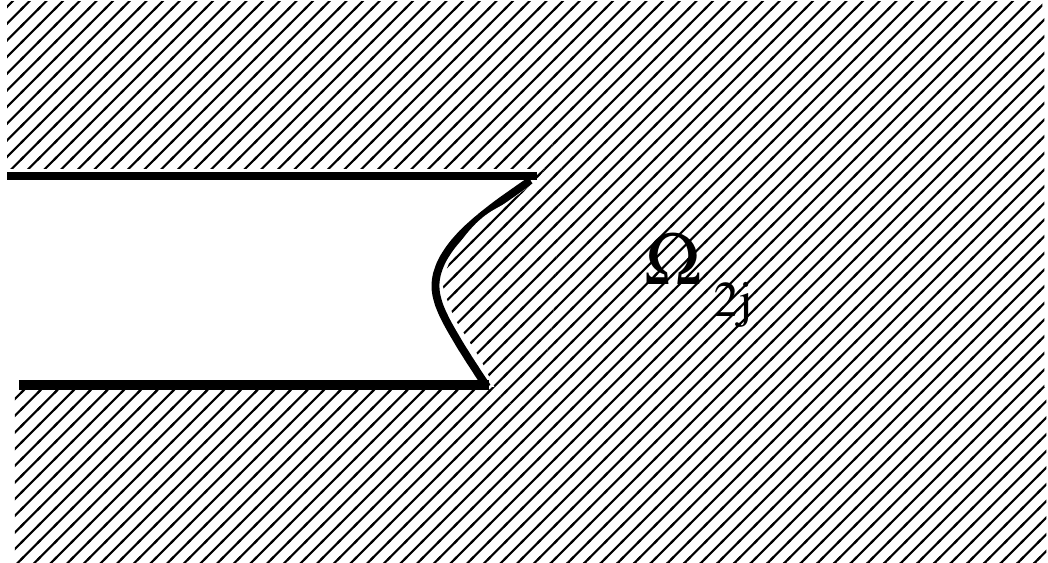}
\label{F.pair} 
} \hskip18mm
\subfigure[$k$ odd]{
\includegraphics[scale=0.3]{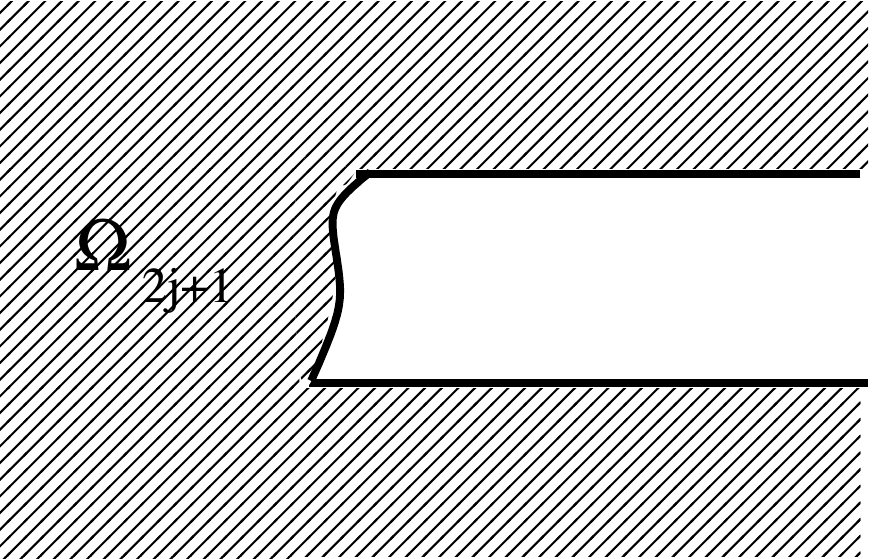}
\label{F.impair} 
}
}
\caption{The domains $\wt \Omega_k$}\label{Fig-omegatilda}
\end{figure}

\bigskip

 By the equalities in \eqref{Eq.phi},
we can take  $R_2$ in \eqref{Eq.Open}  large enough so that
\begin{equation}\label{Eq.Big}
 \frac{1}{2}\rvert z \rvert^m < \sqrt{\rvert \phi (z)\rvert}
<2\rvert z \rvert^m
\end{equation}
when $ \rvert z \rvert \geq R_2$. With this choice, we can
prove the following result.

\begin{lemma}\label{L.Comparaison.2}
 There is a real constant $c_1>0$ such that, for any $z$ satisfying
$ \rvert z \rvert>2R_2$, there exists $k\in \{0,\dots,2m+1\}$ such that
\begin{equation*}
z\in \Omega_k \quad \text{and} \quad
d_\phi (z,\partial \Omega_k)>c_1  \,\rvert z \rvert,
\end{equation*}
where $d_\phi$ stands for
the distance on $\Omega_k$ with respect to the $\phi$-metric
given by
$\rvert \phi (z)\rvert \, \rvert dz\rvert^2$.
\end{lemma}

\begin{proof}
First assume that $m\geq 1$.  Let $z \in \mathcal{U}$ such that
 $ \rvert z \rvert \geq 2R_2$. We choose the determination of
 the argument of $z$ in the interval $[0,2\pi[$.

 Recall that $\alpha_k=\frac{k\pi}{m+1}$ for
$k=-1,\dots,2m+2$. There exists a unique $k\in
\{0,\dots,2m+1\}$
such that either $(\alpha_{k}+\alpha_{k+1})/2 \leq \arg z < \alpha_{k+1}$
or
$\alpha_{k} \leq \arg z < (\alpha_{k}+\alpha_{k+1})/2$. Without loss of
generality, we can assume that the latter occurs.
Therefore $z\in \Omega_{k}.$
\begin{figure}[!ht]
\centerline{\includegraphics[scale=0.3]{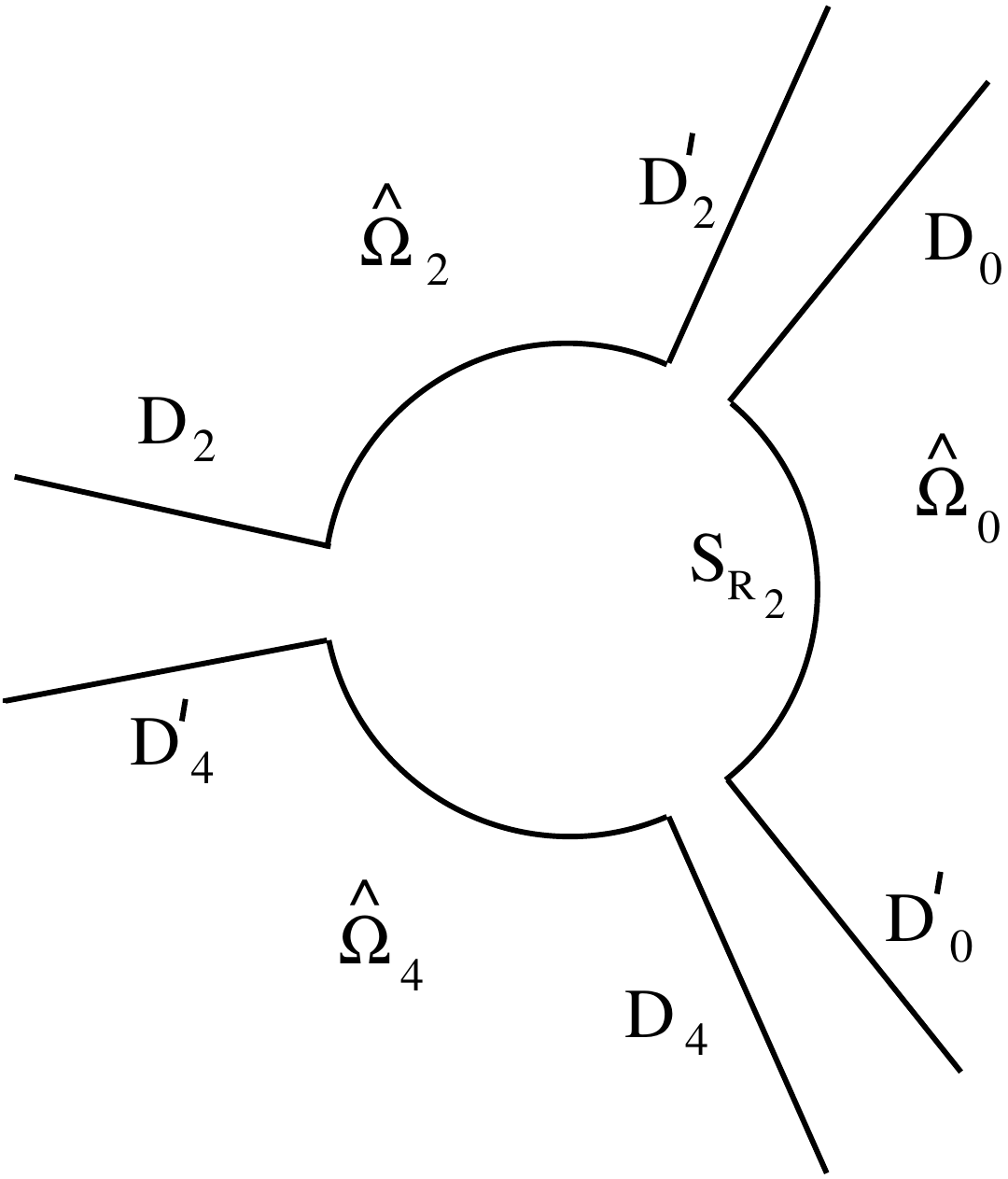}}
\caption{The domains  $\wh \Omega_k$ for $m=2$ and $k$ even}
\label{Fig-chapeau}
\end{figure}

 For any $k=0,\cdots,2m+1$, we
define the following rays:
\begin{align*}
 D_k & := \{ \rho e^{i( \alpha_{k+1} -\pi/10(m+1))},\ \rho\geq R_2\}, \\
 D_k^\prime & := \{ \rho e^{i( \alpha_{k-1} +\pi/10(m+1))},\ \rho\geq R_2\}.
\end{align*}
By assumption, $z$ belongs to the subdomain $\wh \Omega_k$ of $\Omega_k$ bounded
by $D_k$,
$D_k^\prime$ and the arc $\Gamma (R_2)$  of $S_{R_2}$ corresponding to
$ \alpha_{k-1} +\pi/10(m+1) \leq \theta \leq \alpha_{k+1} -\pi/10(m+1)$
(see Figure \ref{Fig-chapeau}).

Then, we have
\begin{equation*}
 d_\phi (z,\partial \Omega_k) \geq \min \big\{ d_\phi(z,D_k),\
d_\phi(z,D_k^\prime),\ d_\phi(z, \Gamma (R_2))\big\}.
\end{equation*}
Let $\gamma :[0,1] \rightarrow \Omega_k$ be any smooth arc satisfying
$\gamma (0)=z$, $\gamma (1) \in D_k$ and $ \rvert \gamma (t)\rvert >R_2$ for
any
$t\in [0,1]$. Denoting by $L_\phi (\gamma)$ the length of $\gamma$ for the
$\phi$-metric and using (\ref{Eq.Big}), we have:
\begin{equation*}
 L_\phi (\gamma) =\int_0^1 \sqrt{\rvert \phi (\gamma(t))\rvert}\,
\rvert \gamma^\prime (t)\rvert\, dt  \geq
\frac{1}{2} \int_0^1 \rvert \gamma (t)\rvert ^m\rvert \gamma^\prime (t)\rvert\,
dt \geq \frac{R_2^m}{2} \int_0^1 \rvert \gamma^\prime (t)\rvert\, dt,
\end{equation*}
so that $ L_\phi (\gamma) \geq (1/2) R_2^m L(\gamma)$, where $L(\gamma)$ is the
Euclidean length of $\gamma$. Since \newline
$4\pi/10(m+1) \leq \alpha_{k+1}-\pi/10(m+1) -\arg z \leq
9\pi/10(m+1)<\pi/2$, \newline
we get
\begin{equation*}
 L(\gamma) > d(z,D_k) >\sin\left( \frac{4\pi}{10(m+1)}\right)\,
\rvert z\rvert
\end{equation*}
where $d(z,D_k)$ stands for the Euclidean distance between $z$ and $D_k$. From
the last inequality,
we
deduce
\begin{equation}\label{Eq.Estimate.D}
 d_\phi(z,D_k) \geq  \frac{R_2^m}{2}  \sin\left( \frac{4\pi}{10(m+1)}\right)\,
\rvert z\rvert.
\end{equation}

\

Let now $\gamma :[0,1] \rightarrow \Omega_k$ be any smooth arc satisfying
$\gamma (0)=z$, $\gamma (1) \in D_k^\prime$ and $ \rvert \gamma (t)\rvert >R_2$
for any $t\in [0,1]$. In the same way, we can show that
\begin{equation*}
  L_\phi (\gamma) \geq \frac{R_2^m}{2}   L(\gamma) \geq \frac{R_2^m}{2}
d(z,D_k^\prime).
\end{equation*}
Since $9\pi/10(m+1) \leq \arg z -\alpha_{k-1}-\pi/10(m+1)\leq 14\pi/10(m+1)$ we
obtain
\begin{equation}\label{Eq.Estimate.D2}
  d_\phi(z,D_k^\prime) \geq \min\big\{ \sin\big( \frac{9\pi}{10(m+1)}\big),
 \sin\big( \frac{14\pi}{10(m+1)}\big)\big\}\, \frac{R_2^m}{2}\,
\rvert z\rvert.
\end{equation}
Finally, let  $\gamma :[0,1] \rightarrow \Omega_k$ be any smooth arc satisfying
$\gamma (0)=z$, $\rvert \gamma (1) \rvert =R_2$ and
 $ \rvert \gamma (t)\rvert >R_2$ for any $t\in [0,1\mathopen[$.
As before we have
\begin{equation*}
L_\phi (\gamma) \geq \frac{R_2^m}{2} L(\gamma) \geq
\frac{R_2^m}{2} \big( \rvert z\rvert -R_2\big).
\end{equation*}
Since $\rvert z\rvert   >2R_2$ we get
\begin{equation}\label{Eq.Estimate.S}
 d_\phi(z, S_{R_2}) > \frac{R_2^m}{4}\, \rvert z\rvert.
\end{equation}
Using estimates (\ref{Eq.Estimate.D}), (\ref{Eq.Estimate.D2}) and
(\ref{Eq.Estimate.S}), we are done in the case $m\geq 1$.

\

Now we consider the case $m=0$. Then, there are only two
domains: $\Omega_0, \Omega_1$, and we have $\alpha_0=0$, $\alpha_1=\pi$ and
$\alpha_2=2\pi$.

Let  $z \in \mathcal{U}$ such that
 $ \rvert z \rvert \geq 2R_2$. For some $k\in \{0,1\}$, we have
either $(\alpha_{k}+\alpha_{k+1})/2 \leq \arg z < \alpha_{k+1}$ or
$\alpha_{k} \leq \arg z <(\alpha_{k}+\alpha_{k+1})/2$. Without loss of
generality, we can assume that the former occurs and that $k=0$, that is:
$\pi/2 \leq \arg z < \pi$ and, therefore, $z\in \Omega_1$.

 We set
\begin{equation*}
 D  := \{ \rho e^{  i\pi/10},\ \rho\geq R_2\}, \qquad
 D^\prime  := \{ \rho e^{ - i\pi/10},\ \rho\geq R_2\}.
\end{equation*}
We have $d(z,D)\geq \rvert z \rvert /2$ and
$d(z,D^\prime)\geq d(z,D)$. Moreover, it can be shown in the same way as in the
case $m\geq 1$, that $d_\phi(z, S_{R_2}) >  \rvert z \rvert /4$. We obtain
that $d_\phi (z,\partial \Omega_k) > \rvert z \rvert /4$, which concludes the
proof.
\end{proof}

\

\begin{remark}\label{R.definition-inverse}
 For $k=0,\dots,2m+1$, the map $W_k : \Omega_k \rightarrow \wt \Omega_k$ is a
conformal diffeomorphism. Since $W_k^\prime (z)=\sqrt{\phi (z)}$, $W_k$
is an isometry when $\Omega_k$ is equipped with the $\phi$-metric
$\rvert \phi (z)\rvert \, \rvert dz\rvert^2$ and $\wt \Omega_k$ is equipped
with the Euclidean metric  $  \rvert dw\rvert^2$.

 We denote by $Z_k : \wt \Omega_k \rightarrow  \Omega_k$ the inverse
function of $W_k$.
\end{remark}

\

\centerline{\sc The  image of the level sets of ${\rm Im}\,W$ by the harmonic
map $F$}

\

\hspace{.1cm}  Let $N:=(n_1,n_2,n_3)$ be the unit normal vector field along the end $E$
such
that
$(X_x,X_y,N)$ has the positive orientation. We get from
\cite[Proposition~ 4]{ST} that
$n_3= \frac{\rvert F_z\rvert -\rvert F_{\ov z}\rvert }
{\rvert F_z\rvert + \rvert F_{\ov z}\rvert }$. We define a function
(possibly with poles)  $\omega$ on $U$ setting  \cite[Formula~ 14 ]{HST}

\begin{equation}\label{F.n3}
  n_3= \tanh \omega.
\end{equation}
For $k=0,\dots,2m+1$, we denote the restriction
of $\omega$ to $\Omega_k$ by $\omega_k $. The function
$\wt \omega_k : \wt \Omega_k \rightarrow \R$ is defined by setting
$\wt \omega_k (w):= (\omega_k \circ Z_k)(w)$ for any $w\in \wt\Omega_k$.

The induced metric $ds^2$ on $U$
reads as
\begin{equation}\label{F.metrique}
 ds^2 = 4\ch^2 (\omega) \, \rvert \phi \rvert \, \rvert dz \rvert^2,
\end{equation}
see \cite[Equation 14]{HST}.

\begin{remark}\label{analytic-rem}
Since $\phi$ has no zero on $U$, the function $\omega$ has no pole
and the tangent plane of $E$ is never horizontal.  This means that the end $E$ is transversal to any slice ${\mathbb H}^2\times \{t\}.$
Thus, the
intersection of  $E$ with any slice is composed of analytic curves.
\end{remark}

\hspace{.1cm} Let us denote by $\Delta_z$ (resp. $\Delta_w$) the laplacian
restricted to $\Omega_k$
(resp. $\wt\Omega_k$) for $k=0,\dots,2m+1$,
with respect to the Euclidean metric $\rvert dz \rvert^2$
(resp. $\rvert dw \rvert^2$). Since
$\Delta_z \omega_k=2\sh(2\omega_k)\, \rvert \phi (z)\rvert$ (see
\cite[Equation 13]{HST}), we
deduce
\begin{equation}\label{Eq.Bochner}
 \Delta_w \wt \omega_k = 2\sh(2\wt \omega_k).
\end{equation}

\bigskip

For any $w\in \wt \Omega_k$ we denote by $d_k(w)$ the Euclidean distance
between $w$ and the boundary of $\wt \Omega_k$.

\

The   following estimate \eqref{Eq.Estimate} can be found in \cite{HR}
 (see also \cite[Lemma 3.3]{Minsky}).

\begin{prop}\label{P.Estimate}
There exists a
constant $K_0>0$ such that for $k=0,\dots,2m+~1$ and for
any $w\in \wt \Omega_k $ with
$d_k(w) >1$, we have
\begin{equation}\label{Eq.Estimate}
 \rvert \wt \omega_k (w) \rvert \leq \frac{K_0}{\ch d_k (w)} < 2K_0 e^{-d_k(w)}.
\end{equation}
Consequently, the tangent planes to the end become vertical at infinity.
\end{prop}
The last assertion is a consequence of the estimate (\ref{Eq.Estimate}),
Lemma \ref{L.Comparaison.2}, and Remark \ref{R.definition-inverse}.

\hspace{.1cm} We recall that the
energy density of the harmonic function $F$ with respect to the
metric $ \rvert \phi (z)\rvert \, \rvert dz\rvert^2  $ on $U$ and the hyperbolic
metric   on $\hi2$
is the real function defined on $U$ by
\begin{equation*}
 e(z):=\frac{(\sigma \circ F)^2(z)}{\rvert \phi (z)\rvert }
\big(\rvert F_z\rvert^2 + \rvert F_{\ov z}\rvert^2\big)
\end{equation*}
Then one has
\begin{equation*}
e(z)= \frac{\rvert F_z\rvert}{\rvert F_{\ov z}\rvert} +
\frac{\rvert F_{\ov z}\rvert}{\rvert F_{z}\rvert}=2 \ch 2\omega.
\end{equation*}
where the  first equality follows from the definition of $\phi$ and the second equality
follows
from the definition of $\omega$.
Observe that $e^{2\omega}= \frac{\rvert F_z\rvert}{\rvert F_{\ov z}\rvert}$.

For $k=0,\dots,2m+1$ we denote by $\wt F_k$ the harmonic map
$\wt F_k := F \circ Z_k : \wt \Omega_k \rightarrow \hi2$.

 Recall that the relation between  the coordinate $z$ in $\Omega$ and  the coordinate $w$
 in  $\wt \Omega$ is  $w=W(z)$ and
$\frac{d z}{d w}=1/\sqrt{\phi \circ Z}$.  The
energy density  $\wt e$  of $\wt F$ with respect to the Euclidean
metric $\rvert dw\rvert^2  $ on $\wt \Omega$ and the hyperbolic
metric   on $\hi2$
is defined on  $\wt \Omega$
by
\begin{equation*}
 \wt e(w):=(\sigma \circ \wt F)^2(w)
\big(\rvert \wt F_w \rvert^2 + \rvert \wt F_{\ov w}\rvert^2\big)
\end{equation*}
As before, we have
\begin{equation*}
\wt e(w)= \frac{\rvert \wt F_w\rvert}{\rvert \wt F_{\ov w}\rvert} +
\frac{\rvert \wt F_{\ov w}\rvert}{\rvert \wt F_{w}\rvert}=2 \ch 2\wt\omega.
\end{equation*}
Thus $ \wt e(w)= e(z)$, if $z=Z(w)$.

\begin{definition}\label{D.courbure}
 Let  $\kappa : U \rightarrow\R$ be
 defined as follows: for any
$z_0\in U$, $\kappa (z_0)$ is the {\em geodesic curvature} in $\hi2$ (with respect to
the normal orientation induced by the unit normal vector field $N$  on $E$)
of the  connected component of  $F(\{ \Im \, W= \Im\, W(z_0)\})$ passing
through the point
$F(z_0)$.

For any $k=0,\dots, 2m+1$, let
$\wt \kappa:  \wt \Omega_k\rightarrow\R$  be defined by setting $\wt \kappa(w_0)=\kappa(z_0),$ where $w_0=W(z_0)$
\end{definition}

As a consequence of Remark \ref{analytic-rem}, we have that the
function $\kappa$ is analytic

\

\begin{lemma} Fix a number $k\in\{0,\dots2m+1\}$  and
consider  the simply connected domain $\Omega_k$ defined in \eqref{Eq.Open}.
Then, setting $w=u+iv$ on $\wt \Omega_k$,
the pullback  by the
harmonic map $\wt F_k : \wt \Omega_k \rightarrow \hi2$ of the hyperbolic
metric $ \sigma^2 (\xi) \rvert d\xi \rvert ^2$ is given by
\begin{align}\label{Eq.Metrique}
 \wt F_k ^* \big(\sigma^2 (\xi) \rvert d\xi \rvert ^2\big) = 4\ch^2 \wt \omega_k \ du^2 +  4\sh^2 \wt \omega_k \ dv^2 .
\end{align}

Moreover, for any horizontal coordinate curve $\wt \gamma:=\{ v=const\}$ in
$\wt \Omega_k$, the absolute value of the geodesic curvature
$\wt \kappa$
of the curve
$\wt F_k(\wt \gamma)$ in $\hi2$ is given by
\begin{equation}\label{Eq.geodesic-curvature}
 \rvert \wt \kappa (w)  \rvert = \frac{1}{2 \ch \wt \omega_k}
\left\rvert \frac{\partial \wt\omega_k}{\partial v}\right\rvert (w)
\end{equation}
for any $w\in \wt \gamma$.
\end{lemma}

\begin{proof}
A straightforward computation shows that
\begin{equation*}
 F_k^* \big(\sigma^2(\xi) d\xi d\ov \xi\big)=
\phi (z) dz^2 + \ov \phi (z) d\ov z ^2 +e(z)\,\rvert \phi (z)\rvert dz d\ov z .
\end{equation*}
Since
$dw=\sqrt{\phi (z)}dz$ and $d\ov w= \sqrt{\ov \phi (z)}d\ov z$, in the
coordinate $w=u+iv$, we have
\begin{equation*}
  \wt F_k^* \big( \sigma^2(\xi) d\xi d\ov \xi\big)=
(\wt e+2)du^2 + (\wt e-2)dv^2
= 4\ch^2 \wt \omega_k \ du^2 +  4\sh^2 \wt \omega_k \ dv^2 .
\end{equation*}

Then equality \eqref{Eq.Metrique} is proved.

Now, let  $w_0\in \wt \gamma$ and assume $\wt\omega_k (w_0)\not=0$.
Then,
by (\ref{Eq.Metrique}),
the pullback by $\wt F_k$ of the hyperbolic metric is a regular metric
in a neighborhood of $w_0$ in  $\wt \Omega_k$. Consequently,  the geodesic
curvature of $\wt F_k(\wt \gamma)$ at $w_0$  is given by
\begin{align*}
\wt \kappa (w_0) &= -\frac{1}{2} \,\frac{1}{4\ch^2 \wt \omega_k}
\,\frac{1}{2\rvert \sh \wt \omega_k \rvert}
\,\frac{\partial }{\partial v}(4\ch^2 \wt \omega_k)(w_0)\\
&=-\frac{1}{2} \frac{1}{\ch \wt \omega_k}
\frac{\sh \wt \omega_k}{\rvert \sh \wt \omega_k \rvert}
 \frac{\partial \wt \omega_k}{\partial v} (w_0),
\end{align*}
(see \cite[Formula (42.8)]{Kreyszig}).
Therefore,  the proof is finished in the case
$\wt \omega_k (w_0)\not=0$.

\

Assume now that $\wt \omega_k (w_0)=0$. If $\wt \omega_k$ vanishes
identically in a neighborhood of $w_0$, then the tangent plane of the minimal
end $E$ is always vertical in a open neighborhood of $X\big(Z_k(w_0)\big)$.
This
means that such  a neighborhood
 is contained in a vertical cylinder in
$\hi2 \times \R$. Since $E$ is minimal, the vertical cylinder is a part of a
vertical geodesic plane and, by analyticity, the whole end $E$ is contained
in the geodesic plane. Consequently the curve $\wt F_k(\wt \gamma)$ is
a part of a geodesic of $\hi2$ and formula
(\ref{Eq.geodesic-curvature})
is trivially satisfied.

If $\wt \omega_k$ is not identically zero in a neighborhood of $w_0$, then there
exists a sequence $(w_n)_{n\in \n^*}$ in $\wt \Omega_k$ converging to $w_0$ such
that $\wt \omega_k (w_n) \not=0$ for any $n>0$.
Since formula  (\ref{Eq.geodesic-curvature})
 holds at any point $w_n$ and  $|\kappa|$ is a continuous function, then
\eqref{Eq.geodesic-curvature} holds also at $w_0$.
\end{proof}

The following Proposition is crucial  in order to understand the geometry of the horizontal sections.

\begin{prop}\label{P.courbure}
Let $z_0\in U$ and let $\kappa (z_0)$ be the geodesic curvature of the level
curve $F(\{\Im\, W(z)=\Im\, W(z_0)\})\subset \hi2.$ We set $R_3=\max \{2R_2, 2/c_1\}$, where $c_1>0$ is the constant given by Lemma
$\ref{L.Comparaison.2}$.
Then,
there exists a constant $c_2 >0$ such that,
for any $z_0\in U_{R_3},$ we have
\begin{equation*}
 \rvert \kappa (z_0) \rvert < c_2 e^{-c_1\rvert z_0\rvert}.
\end{equation*}
\end{prop}

\begin{proof}
 Let $z_0 \in U$ be any point such that
$\rvert z_0 \rvert > R_3$. It follows from Lemma \ref{L.Comparaison.2} that
there exists $k\in \{0,\dots,2m+1\}$ such that
\begin{equation*}
 z_0 \in \Omega_k \quad \text{and}\quad
d_\phi (z_0, \partial \Omega_k)>c_1 \rvert z_0\rvert.
\end{equation*}
Setting $w_0 := W(z_0)\in \wt \Omega_k$, we get
$d_k (w_0):=d_k(w_0, \partial \wt\Omega_k) = d_\phi (z_0, \partial \Omega_k)$,
where $d_k$ denotes  the Euclidean distance in $\wt \Omega_k$.
Therefore, we obtain
$d_k(w_0) >2,$ since
$c_1 \rvert z_0\rvert > c_1 R_3 >2$.

Let $\wt D$ be the unit disk in the $w$-complex plane, centered at $w_0$, thus
$\wt D \subset \wt \Omega_k$.
For any $w\in \wt D,$ we denote by $d(w)$ the Euclidean distance between
$w$ and $\partial \wt D$. Recall that the function $\wt \omega_k$ satisfies
Equation  (\ref{Eq.Bochner})  on $\wt \Omega_k$. We restrict $\wt\omega_k$ to $\wt D$ and we apply the interior {\em a-priori} gradient estimate  for the Poisson Equation \cite[Theorem 3.9]{G-T}, then

\begin{equation*}
 \sup_{\wt D} \big( d(w) \rvert \nabla \wt \omega_k\rvert \big)   <
K_1 \big( \sup_{\wt D} \rvert \wt \omega_k\rvert +
2\sup_{\wt D} d^2 (w) \rvert  \sh 2\wt \omega_k\rvert \big),
\end{equation*}
for some constant $K_1 >0$, where $\nabla$ means the Euclidean gradient.

Since $d(w_0)=1$ and $d(w)\leq 1$ for any $w\in \wt D$, we get
\begin{equation*}
  \rvert \nabla \wt \omega_k\rvert (w_0) < K_1 \big( \sup_{\wt D} \rvert \wt
\omega_k\rvert +
2\sup_{\wt D}  \rvert  \sh 2\wt \omega_k\rvert \big)
\end{equation*}
Moreover, since
$d_k( w) \geq d_k (w_0)-1$ for any $w\in \wt D$, we deduce
from Proposition \ref{P.Estimate} that
\begin{equation*}
  \rvert \wt \omega_k (w) \rvert \leq
\frac{K_0}{\ch (d_k (w_0)-1)}
\end{equation*}
for any $w\in \wt D$. Using the inequality $\ch (t-1)>\frac{e^t}{10}$ for any
$t\in \R$ we obtain
\begin{equation*}
  \rvert \nabla \wt \omega_k\rvert (w_0) < K_1 \big(
10 K_0 e^{-d_k(w_0)} + 2 \sh ( 20K_0 e^{-d_k(w_0)})\big).
\end{equation*}
The function $x\mapsto \frac{\sh x}{x}$ is strictly increasing
for $x>0.$ As $d_k(w_0)>2$, then we obtain that
$\sh ( 20K_0 e^{-d_k(w_0)}) < e^2 \sh (20K_0e^{-2})\, e^{-d_k(w_0)}$.
This proves that there exists a constant $\delta >0$ such that
\begin{equation}\label{F.estimee du gradient}
   \rvert \nabla \wt \omega_k\rvert (w_0) < \delta\, e^{-d_k(w_0)}
\end{equation}
for any $w_0\in \wt \Omega_k$ such that $d_k(w_0)>2$.
From formula (\ref{Eq.geodesic-curvature})
and from the previous
computations, setting $c_2:=\delta/2$, we conclude that
\begin{equation*}
 \rvert \wt \kappa (w_0) \rvert < c_2 e^{-d_k(w_0)}.
\end{equation*}
As $\kappa (z_0)=\wt \kappa (w_0)$ and
$d_k(w_0):=d_\phi (z_0,\partial \Omega_k)>c_1 \rvert z_0\rvert$
by Lemma \ref{L.Comparaison.2}, this completes the proof.
\end{proof}

\

In view of Lemma \ref{L.Level set}, let us sum up some notations previously
established in the sequence of Corollary \ref{C.Ouverts modeles}.
For any $C\in \R$ and for
$k=0,\dots,2m+1$, $H_k(C,R)\subset \mathcal{U}\subset U $,
denotes
the
semi-complete level curve of the function $ \Im\, W$ whose
asymptotic direction is $\{re^{i\alpha_k},\ r>0\}$, where
$\alpha_k=k\pi/(m+1)$. That is
\begin{equation*}
 \Im\, W(z)=C\ \text{for any } z\in H_k(C,R)  \text{ and }
\arg z \xrightarrow[\rvert z\rvert \to \infty ]{}    \alpha_k,  \ \ z\in H_k(C,R).
\end{equation*}


For any $C>C_0$, we denote by $L_j(C)$ (resp. $L_j(-C)$), $j=0,\dots,m$,
the proper and complete level curves given by $\{\Im\, W=C\}$ (resp. $\{\Im\,
W=-C\}$).
We have, for any $R,$
$H_{2j}(C,R)\cup H_{2j+1} (C,R)\subset L_j(C)$ and
$H_{2j+1}(-C,R)\cup H_{2j+2} (-C,R)\subset L_j(-C)$, $j=0,\dots,m$, where
$H_{2m+2}(-C,R):=H_0(-C,R)$.

\

The notion of {\em convergence in the $C^1$ topology } in the next Theorem is given in the
statement of Definition \ref{D.C1}.

\begin{theo}\label{T.Description}

\begin{enumerate}
\item \label{description.1} For any  $C\in \R,$ let  $r(C)$ be  defined  as in
Lemma
\ref{L.Description}.
Then, for any   $C\in \R,$  for any $k\in
\{0,\dots,2m+1\}$ and $R>r(C),$ the level curve
$F(H_k(C,R))\subset \hi2$ is a proper semi-complete curve
which has no limit
point in
$\hi2$ and
with a unique asymptotic point in $\pain \hi2$.

\smallskip

\item \label{description.2} For any $C_1,C_2\in \R$ and for any
$k\in \{0,\dots,2m+1\}$,
the level curves
$F(H_k(C_1,R))$, $F(H_k(C_2,R)) \subset \hi2$ are asymptotic.
{   More precisely, for any $\varepsilon >0$ there is a compact subset
$K\subset \hi2$ such that
for any $C$ between $C_1$ and $C_2$
 the level curve $F(H_k(C,R))\setminus K$
remains  in a $\varepsilon$-neighborhood of  $F(H_k(C_1,R))\setminus K$ .}

Consequently, $F(H_k(C_1,R))$ and $F(H_k(C_2,R))$
have the same
asymptotic point
$\theta_k \in~ \pain\hi2$.

\smallskip

\item \label{description.3} For $k=0,\dots,2m+1$, the asymptotic points
$\theta_k$
and
$\theta_{k+1}$ are distinct,
$( \theta_{2m+2}:=~\theta_0)$.

\smallskip

\item \label{description.4} Let $j\in \{0,\dots,m\}$.
\begin{itemize}
\item When $C\to +\infty$, then the proper and complete level curves
$F(L_j(C))\subset \hi2$ converge for the  $C^1$ topology
 to the geodesic in $\hi2$ with asymptotic
boundary $\{ \theta_{2j},  \theta_{2j+1}\}$.

\item  When $C\to +\infty$, then the proper and complete level curves
$F(L_j(-C))\subset \hi2$ converge for the $C^1$ topology to the geodesic in
$\hi2$ with asymptotic
boundary $\{ \theta_{2j+1},  \theta_{2j+2}\}$,
$(\theta_{2m+2}:=~\theta_0)$.
\end{itemize}
\end{enumerate}
\end{theo}

\begin{proof}
Assertion (\ref{description.1}) is a straightforward consequence of the
curvature
estimates given in
Proposition \ref{P.courbure}, together with Proposition \ref{P.limit}.

\
\medskip

Let us prove Assertion (\ref{description.2}). Let
$k\in \{0,\dots,2m+1\}$. By Lemma \ref{L.Level set}, for any $R>r(C_i),$  we have
$H_k (C_i,R)\subset \Omega_k$, $i=1,2$.
From formula (\ref{Eq.Reel}) we deduce that
$\Re\, W (z) \xrightarrow[\rvert z\rvert \to \infty]{}
+ \infty,$ $z\in H_k (C_i,R)$ (resp.  $-\infty$) if $k$ is an even (resp. odd) number,
$i=1,2$.

Assume now that $k$ is even (the argument is analogous in  the other case).
{   Then, by the geometry of the sets  $\wt \Omega_k= W(\Omega_k)$
(see Corollary \ref{C.Description}), setting $\wt p_u:= u+iC_1$ and
$\wt q_u:=u+iC_2$,
for any real number $u >0$ large enough, we have
$\wt p_u \in W\big(H_k(C_1)\big)\subset \wt \Omega_k$ and
$\wt q_u \in W\big(H_k(C_2)\big) \subset \wt \Omega_k$. Moreover,
setting  $\wt\gamma_u := \{ (1-t)\wt p_u +t \wt q_u,\ 0\leq t\leq 1\}$,
we have $\wt\gamma_u \subset \wt \Omega_k$.

Let us set $p_u=Z_k (\wt p_u)$, $q_u=Z_k (\wt q_u)$ and
$\gamma_u=Z_k (\wt \gamma_u)$, where
$Z_k : \wt \Omega_k \rightarrow \Omega_k$ is the inverse function of
$W$
restricted to $\Omega_k$ {\rm as defined in  Remark
\ref{R.definition-inverse}.}
Thus, we have:
\begin{itemize}
\item $\partial \gamma _u =\{p_u,q_u\}$, $p_u\in H_k(C_1,R)$ and
$q_u\in H_k(C_2,R)$.
\item $ \Re\, W(z) =u,$ for any $z\in \gamma_u.$
\end{itemize}
The distance between $ F_k(p_u)$
and $ F_k(q_u)$ in $\hi2$ is
smaller than the length of $ F_k( \gamma_u)$ which, by construction, is
equal to the length of $\wt F_k(\wt \gamma_u)$.

We need to prove the following Claim.

\medskip

 {\em {\bf Claim.} Let $\gamma \subset \gamma _u$ be an
open arc along
which the restriction of $\omega$ to $\Omega_k$ vanishes.
Then, $ F_k(\gamma) \subset \hi2$ is reduced to a single
point.}

\vskip2mm

 Indeed, let
$c: \mathopen]0,1\mathclose[\rightarrow \gamma \subset \gamma_u$ be a smooth
parametrization of $\gamma$. Since $\Re\, W \big(c(t)\big) \equiv u$ for
$0<t<1$, setting $w=W(z)$ and differentiating with respect to $t$, we
get
$\frac{d w}{dz }\, \frac{d c}{dt} + \ov{ \left( \frac{d w}{dz }\right)}\,
  \ov{ \left(\frac{d c}{dt}\right)}=0$. Moreover, as $\omega _k (c(t))\equiv 0$
we have $\rvert F_z\rvert =\rvert F_{\ov z}\rvert$ along $\gamma$. Recall that
$\frac{d w}{dz }=\sqrt{\phi}$ and
$\phi = (\sigma \circ F)^2 F_z \ov F _z$.
 Combining
those  relations, we obtain that
$F_z\, \frac{d c}{dt} + F_{\ov z}\, \ov{ \left(\frac{d c}{dt}\right)}\equiv 0.$
Then $\frac{d}{dt} (F\circ c)(t) \equiv 0$, which proves the Claim.

\

Since the function $\omega$ is real analytic, its restriction to the analytic
arc $\gamma_u$ vanishes identically or has a finite number of zeroes.
Consequently, the length
of $ F_k(\gamma_u)$ in $\hi2$ is equal to the length of $\wt \gamma_u$ with
respect to the pseudo-metric (\ref{Eq.Metrique}) on $\wt \Omega_k$,
denoted by $L_k (\wt \gamma_u)$.
 Corollary \ref{C.Description} yields that for any
$w\in \wt \gamma_u$
and for $u$ large enough we have
\begin{equation*}
d_k(w):=d_k(w,\partial \wt \Omega_k) \geq u-a_k,
\end{equation*}
where, as usual, $d_k$ is the Euclidean metric on $\wt \Omega_k$.
We have
\begin{equation*}
d_{\hi2} \big( F_k(p_u), F_k(q_u)\big) =
d_{\hi2} \big(\wt F_k( \wt p_u),\wt F_k( \wt q_u)\big)  \leq
L_{\hi2} \big( \wt F_k(\wt \gamma_u)\big)= L_k (\wt \gamma_u)
\end{equation*}

 where $d_{\hi2}$ (resp. $L_{\hi2}$) is the distance (resp. the length) in the hyperbolic metric.

As $\Re\, W \equiv u$ along $\wt \gamma_u$, we obtain
\begin{align*}
  L_k (\wt \gamma_u) &= 2\rvert C_1 -C_2\rvert \int _0^1
\rvert \sh \wt \omega_k \big(\wt \gamma_u (t)\big)\rvert \, dt \\
  &\leq 2 \sh \left( \frac{K_0}{\ch (u-a_k)}\right) \rvert C_1 -C_2\rvert,
\end{align*}
where the
inequality comes from formula~ (\ref{Eq.Estimate}). Hence,  we have
\begin{equation*}
 d_{\hi2} \big( F_k(p_u), F_k(q_u)\big) \to 0 \quad
\text{when}\quad u\to +\infty.
\end{equation*}
This completes the proof of
Assertion (\ref{description.2}).}
\

\medskip

Let us prove Assertion (\ref{description.3}). Assume, for instance, that
$\theta_0= \theta_1$. Then, for any $C>C_0$,  there exists a complete level
curve $L_0(C)\subset \Omega_0$ such that $F(L_0(C)) \subset \hi2$ is a proper
and
complete curve with
$\pain F(L_0(C))=\{\theta_0\}$.  We deduce
from Proposition \ref{P.courbure}
and formula (\ref{Eq.Ima})
that for $C$ large enough, {\rm the absolute value} of the geodesic curvature of
$F(L_0(C)) \subset \hi2$ is smaller than 1/4.
Let $\Gamma \subset \hi2$ be any complete geodesic
 such that $\theta_0\in \pain \Gamma$. Then, we obtain a
contradiction with the maximum principle, comparing  $F(L_0(C))$ with the
family of complete curves $\gamma_p,\ p\in \Gamma$, orthogonal to $\Gamma$
at $p$, with constant curvature 1/2, and such that $\theta_0$ belongs to the
asymptotic boundary of the mean convex component of $\hi2 \setminus \gamma_p$.
This completes the proof of
Assertion (\ref{description.3}).

\medskip

Assertion (\ref{description.4}) is a straightforward consequence of
Assertion (\ref{description.3}),
Propositions \ref{P.courbure} and  \ref{P.C1}.
\end{proof}

\begin{remark}
(1) We deduce from Theorem \ref{T.Description} that the asymptotic boundary of
$F(U)$ is composed of exactly $2m+2$ points, counting  with multiplicity. In
particular, if $m=0$ then  $\pain F(U)$ has exactly two distinct points.

(2) Observe that the $2m+2$ asymptotic points $\theta_0,\dots,\theta_{2m+1}$
of $F(U)$ need not to be distinct. They even not need to be well ordered as we
can see in some examples found by  J. Pyo and M. Rodriguez \cite{P-R}. \newline
We can construct  {\em artificial} examples for which the asymptotic points are not
distinct: just consider the covering maps $\psi_n : U\rightarrow U$, $n \geq
2$,
 defined by $\psi_n (z)=z^n$, and the minimal ends
$X_n:= X\circ \psi_n : U \rightarrow \hi2 \times \R$.

\end{remark}

We will give an alternative  geometric interpretation of Theorem \ref{T.Description} in terms of polygonal curves.
In order to to this, we need some definitions.

The asymptotic boundary  $\partial_{\infty}( \hi2 \times \R)$ is topologically equivalent to the following open  cylinder joint with two closed
disks:

$$ {\mathcal C}=\{\s ^1 \times (-1,1)\}\cup D(+1) \cup D (-1)$$
where $D(-1)=\{ u \in \C; |u| \leq 1\} \times \{-1\}$ and $D(+1)=\{ u \in \C; |u| \leq 1\} \times \{+1\}$.
We identify   $int (D (+1))$ and $int (D (-1))$ with the hyperbolic plane.
Let  $t : (-1,+1) \to \R$ be a homeomorphism.  For any $y\in(-1,1),$  we identify
 $\s^1 \times \{y\}$ with the asymptotic boundary of  $ \hi2 \times \{t(y)\}.$
The sets $int (D (+1))$ and $int (D (-1))$ represent  the closure of vertical geodesics $\{p\}\times \R,$ $p\in\hi2.$

\begin{definition}\label{D.Polygone} We say that $\mathcal P$ is a {\em closed
polygonal curve}  if it is a closed curve contained in $\mathcal C,$  that is
union of a finite number of hyperbolic
geodesics in $int (D (+1))$ and $int (D (+1)),$ jointed by vertical segment in  $\s ^1 \times (-1,1),$ joint with their endpoints.
\end{definition}

Notice that the closed polygonal curve $\mathcal P$ may happen  to be not embedded and some of its  sides may have multiplicity greater than one.

Now, we give the promised alternative interpretation of Theorem \ref{T.Description}.

\begin{prop}\label{P.Polygone}
Let $X:=(F,h): U \to \hi2 \times \R$ be a properly immersed finite total curvature end $E=X(U)$, then
$\partial _{\infty} E$ can be identified with a closed  polygonal curve $\mathcal P$ in $\mathcal C,$ where:

\begin{itemize}

\item A geodesic $\gamma_{+1}$  of ${\mathcal P}$  contained  in  $D (+1)$
means that the end $E$ contains a topological half plane which is
asymptotic to $\gamma_{+1} \times \R_+$ when $h$ tends
to  $+\infty.$

\item A geodesic $\gamma_{-1} $ of ${\mathcal P}$  contained  in $D (-1)$ means that the
end
$E$ contains a topological half plane which is asymptotic to $\gamma_{-1} \times \R_-$ when $h$ tends
to  $-\infty.$

\item A vertical segment $\{ p \} \times (-1,+1)$ of $ {\mathcal P}$ means
that   $p\times \r$ belongs  to the asymptotic boundary of $E.$
\end{itemize}

\end{prop}

An interesting problem is to determine the correspondence between the space of closed polygonal curve
$\mathcal P$ and the set of finite total curvature ends. We would like  to understand the relation between
the geometry of the end and the geometry of
$\mathcal P$.

\

We remark that embedded ends can be only observed when ${\mathcal P}$ is an embedded  polygonal
curve.  Properties of ${\mathcal P}$ can be derived from its projection
$\pi ({\mathcal P})$  on a horizontal
hyperbolic plane:
$$\pi : {\mathcal P} \to \hi2\times \{0\}.$$

 M. Rodriguez and J. Pyo has constructed an interesting example of
 a properly embedded minimal surface in $\hi2\times\r.$
 The example is simply connected so that it has only one end.
The polygonal  curve ${\mathcal P}$ associated to the end is embedded with non embedded projection $\pi ( {\mathcal P} ).$
The end  has finite total curvature, contains a vertical geodesic
 $\{p \} \times \R,$  and it is  not a graph.

\

\hspace{.1cm} {   Let us state some results that have an independent
interest in this theory.}

\begin{theo}\label{T.propre}
 Let  be a complete minimal
end with finite total curvature. Then, $E$ is properly immersed.
\end{theo}

\begin{proof}
Let $m\in \n$ be the degree of the end $E$ (see  Definition \ref{D.degre}).
Let $(p_n)$ be a sequence in $U$ such that
$\rvert p_n\rvert \to +\infty$. We want to show that $(X(p_n))$ is not a
bounded sequence in $\hi2 \times \R$. Up to choose a subsequence, we can
assume that there is $k\in \{0,\cdots,2m+1\}$ such that, for any $n,$ we have
$p_n \in \Omega_k.$ Recall that $h(z) =2\,\Im\, W(z)$ for any $z$ in $U$.

 If $h(p_n)\to \infty$ we are done.

 Assume that  the sequence  $(h(p_n))$ of real number is bounded.
Thus,  up to considering a subsequence, there exists a real number $C_1$
such that $h(p_n)\to C_1$. We set
$S(C_1) :=\{ z\in \Omega_k \mid  h(z)=C_1\}$. Thus, $S(C_1)$ is either a
complete curve
or a semi-complete curve. As in the proof of Theorem \ref{T.Description}, we
can construct a sequence $(q_n)$ in $\Omega_k$ such that
\begin{equation*}
 \forall n\in \n, \quad q_n \in S(C_1)  \quad \text{and}\quad
d_{\hi2}\big(F(p_n), F(q_n)\big) \to 0.
\end{equation*}
Since $F(S(C_1))$ has no limit point in $\hi2$ and has an asymptotic point
$p_\infty \in \pain \hi2$ (see (1) in Theorem \ref{T.Description}), we get that
$F(q_n) \to p_\infty$ and consequently  $F( p_n) \to p_\infty$.
Therefore, the sequence $X(p_n)$ is not bounded, which concludes the proof.
\end{proof}

\bigskip

\begin{theo} \label{T.estimees de courbure}
{   Let $X:= (F,h) : U \rightarrow E\subset \hi2 \times \R$
be a minimal, complete end with finite total
curvature. Then, there exists a constant $c_3$ such
that, for any $p\in E$,
we have
\begin{equation}\label{F.courbure sur un bout}
 \rvert K_E (p) \rvert \leq c_3 \, e^{-d_E(p,\partial E)},
\end{equation}
where $K_E$ denotes the intrinsic Gauss curvature and $d_E(\cdot, \partial E)$
stands for the intrinsic distance on $E$.}
\end{theo}

\begin{proof}
Let $m\in \n$ be the degree of the end $E$ with respect to the parametrization
$X$. We consider the
open sets $\Omega_k\subset U$,  $k=0,\cdots,2m+1$, as defined in
\eqref{Eq.Open}
and the real number $R_3 >1$ given in
Proposition
\ref{P.courbure}.   In this proof we will use  the notations previously
established
for the function $W=\Im\, \int \sqrt{\phi (z)}\,
dz,$
$\wt\Omega_k=W(\Omega_k),$  and $Z_k:\wt\Omega_k\longrightarrow \Omega_k$.

 Let $p\in E$ and let $z\in U$ such that $X(z)=p$.   Assume   first that
$|z|>R_3$.
 Therefore, by Lemma \ref{L.Comparaison.2},
there exists $k\in \{0,\dots,2m+1\}$ such
that
$z\in \Omega_k$ and $d_\phi (z,\partial \Omega_k) >c_1 |z|$.
We set $w=W(z)$, so that $w\in \wt \Omega_k$.
We deduce from formula (\ref{F.metrique}) that
the metric $\wt ds^2$ induced on $\wt \Omega_k$ by the minimal immersion
$\wt X := X\circ Z_k  : \wt \Omega_k \rightarrow \hi2 \times \R $ is given by
\begin{equation*}
 \wt ds^2 = 4 \ch^2 \wt \omega_k (w)\,  |dw|^2.
\end{equation*}
Therefore, we obtain
\begin{equation*}
 K_E (p)= K_{\wt ds^2} (w)=
-\frac{\tanh \wt \omega_k}{4\ch^2 \wt \omega_k}\, \Delta \wt \omega_k
- \frac{1}{4\ch^4 \wt \omega_k}\, |\nabla \wt \omega_k|^2.
\end{equation*}
It follows from the proof of  Proposition \ref{P.courbure} that
  $d_k(w)>2$. Now, by a straightforward
computation, using
formulas
(\ref{Eq.Bochner}), (\ref{Eq.Estimate}) and
(\ref{F.estimee du gradient}) (where $\delta=2c_2$), we obtain
\begin{equation*}
  \rvert K_{\wt ds^2} (w) \rvert  < c_3 e^{-2d_k (w)},
\end{equation*}
for some constant $c_3>0$ which does not depend on $w$. Observe that
\begin{equation*}
 d_E(p,\partial E)= d_E(\wt X (w), \partial E)\geq
d_E(\wt X (w), \wt X (\partial \wt \Omega_k))
=d_{\wt ds^2} (w,\partial \wt \Omega_k).
\end{equation*}
From the comparison of the metric $d_{\wt ds^2}$ with the Euclidean metric, we
infer
\begin{equation*}
 d_{\wt ds^2} (w,\partial \wt \Omega_k)
\geq 2d_k(w).
\end{equation*}
 Formula \eqref{F.courbure sur un
bout} follows by   the previous inequalities
 for $|z|>R_3$, i.e. outside a compact subset of
the end $E$. Finally,  it suffices to observe that the continuous function
$p\mapsto |K_E (p)| \, e^{d_E(p,\partial E)}$  is bounded on any compact
subset of $E$.
\end{proof}

\begin{remark}
 A straightforward  consequence of Theorem \ref{T.estimees de courbure} is
 the following: for any complete and connected minimal surface
$\Sigma\subset \hi2 \times \R$
with finite total curvature, and for any $p_0 \in \Sigma$, there exists a
constant
$c_4=c_4(p_0,\Sigma)$  such that for any   $p\in \Sigma$ we have
\begin{equation*}
 \rvert K_\Sigma (p)\rvert \leq c_4\, e^{-d_\Sigma(p, p_0)}.
\end{equation*}
\end{remark}

\bigskip

\begin{lemma}\label{L.plonge}
{   Let $X:= (F,h) : U \rightarrow E\subset \hi2 \times \R$ be a complete
minimal end with finite total curvature. Let $m\in \n$, be the degree of the
end $E$.

 For any $k\in \{0,\dots, 2m-1\}$, there exists a compact subset
 $K\subset \co$ such that the
restricted  minimal and
 conformal immersion
 $X=(F,h) : \Omega_k\setminus K \rightarrow \hi2 \times \R$ is an embedding. }
\end{lemma}

\begin{proof}
{   To simplify the notations we give the proof for $k=1$.

Recall that  the immersion $X$ is proper, Theorem \ref{T.propre},  and that
$h=2\, \Im\,  W$.
Therefore we
deduce from Lemma \ref{L.Level set} and Proposition \ref{P.courbure} that there
exists a real number
$C_1 > C_0$ with the property  that for any $C \geq C_1$,
the level set
$\{ \Im\,  W (z) = C\}$ (resp. $\{ \Im\,  W (z) = -C\}$)
in $\Omega_1$ consists of a complete curve $L(C)$ (resp.  $L(-C)$)
 such that the geodesic curvature of $F(L(C))$
(resp. $F(L(-C))$) in $\hi2$ is
smaller than 1/4 in absolute value.

Consequently for any
$C$ satisfying $ C \geq C_1$, we get that $F(L(C))$ and
  $F(L(-C))$ are complete and embedded curves in $\hi2$.
We deduce from Theorem \ref{T.Description} that there exist
 $\theta_0,\theta_1,\theta_2 \in \pain \hi2$, with
$\theta_0 \not= \theta_1$, $\theta_1 \not= \theta_2$, and such that
\begin{equation*}
 \pain F(L(C)) =\{ \theta_0, \theta_1\} \qquad \text{and} \qquad
 \pain F(L(-C))=\{ \theta_1, \theta_2\}.
\end{equation*}
Considering the height function $h$, we deduce that the restriction of $X$
 to the nonconnected subset of $\Omega_1$ bounded by
 $L (C_1) \cup L(-C_1)$ is an embedding.

 Let $\varepsilon >0$.
 We deduce from (the proof of) Theorem \ref{T.Description} -
(\ref{description.2}), that there exist $a^+ \in L(C_1)$,
$a^- \in L(-C_1)$ and a compact arc $\gamma_1 \subset \Omega_1$, joining
$a^+$ and $a^-$ and verifying
\begin{itemize}
\item $\Re\,  W$ is constant along $\gamma_1$.

\item $\Im\,  W$ is strictly monotonous  along $\gamma_1$.

\item Denoting by $L_1(C_1)$ (resp. $L_1(-C_1)$
the component of $L(C_1)\setminus \{a^+\}$
(resp. $L(-C_1)$  $\setminus \{a^-\}$)
with asymptotic direction the ray
$\{re^{i\frac{\pi}{m+1} }  \}$, then $F(L_1(C_1))$ remains in a
$\varepsilon$-neighborhood of $F(L_1(-C_1))$ in $\hi2$ (and also
$F(L_1(-C_1))$  remains in a
$\varepsilon$- neighborhood of $F(L_1(C_1))$). Therefore we have
$\pain F(L_1(C_1))=\pain F(L_1(-C_1))= \theta_1$ (say).

\item For any $C\in [-C_1,C_1]$, denoting by $H_1(C) \subset \Omega_1$ the
semi-complete level curve $\{ \Im\,  W(z)= C\}\cap \Omega_1$ issue from
$\gamma_1$ and with asymptotic direction the ray
$\{re^{i\frac{\pi}{m+1} }\}$,  then $F(H_1(C))\subset \hi2$ remains in a
$\varepsilon$-neighborhood of  $F(L_1(C_1))$ and $F(L_1(-C_1))$.
\end{itemize}
We deduce that the restriction of $X$ to the connected component of $\Omega_1$
bounded by $\gamma_1$ and a part of $L(C_1)\cup L(-C_1)$ and containing
$L_1(C_1)\cup L_1(-C_1)$ is an embedding.

 \medskip

  Recall that $L_0^+ := \{\Im\,  W (z)= C_0\} \cap \Omega_1$
 (resp. $L_0^- := \{\Im\,  W (z)= -C_0\} \cap \Omega_1$) is a complete curve
and
that $\pain F(L_0^+)=\{ \theta_0, \theta_1\}$ and
$\pain F(L_0^-)=\{ \theta_1, \theta_2\}$.

\medskip

We deduce again from (the proof of) Theorem \ref{T.Description} -
(\ref{description.2}), that there exist
$b_0^- \in L_0^-$, $b_1^- \in L(-C_1)$, a compact arc
$\gamma_2\subset \ov \Omega_1$ joining $b_0^-$ and $b_1^-$ verifying
\begin{itemize}
\item $\Re\,  W$ is constant along $\gamma_2$.

\item $\Im\,  W$ is strictly monotonous  along $\gamma_2$.

\item Denoting by $(L_0^-)_2$ (resp. $L_2(-C_1)$
the component of $L_0^- \setminus \{b_0^-\}$
(resp. $L(-C_1)$ $\setminus \{b_1^-\}$)
with asymptotic direction the ray
$\{re^{i\frac{2\pi}{m+1} }  \}$, then $F((L_0^-)_2)$
remains in a
$\varepsilon$-neighborhood of $F(L_2(-C_1))$ in $\hi2$ (and also
$F(L_2(-C_1))$ remains in a
$\varepsilon$- neighborhood of $F((L_0^-)_2)$).

\item For any $C\in [-C_1,-C_0]$, denoting by $H_2(C) \subset \Omega_1$ the
semi-complete level curve $\{ \Im\,  W(z)= C\}\cap \Omega_1$ issue from
$\gamma_2$ and with asymptotic direction the ray
$\{re^{i\frac{2\pi}{m+1} }\}$,  then $F(H_2(C))\subset \hi2$ remains in a
$\varepsilon$-neighborhood of both $F((L_0^-)_2)$ and $F(L_2(-C_1))$.
\end{itemize}

Since $\pain F(L_2(-C_1))=\theta_2$ and
$\pain F(L_1(-C_1))=\theta_1\not=\theta_2$, we may
assume that the points $a^+,a^-,b_0^+$ and $b_0^-$  above are chosen so
that the curves
$F(L_1(-C_1))$ and $F(L_2(-C_1))$ are far away from each other.

Consequently, for any $C\in [-C_1,-C_0]$, we have
$F(H_2(C)) \cap F(H_1(C))=\emptyset$. We deduce that the restriction of $X$ to
the subset of $\Omega_1$ bounded by $(L_0^-)_2$, $\gamma_2$, a compact part of
$L(-C_1)$, $\gamma_1$ and a part of $L(C_1)$, and containing $L_2(-C_1)$
(and $L_1(-C_1)$) is an embedding.

\medskip

In the same way, there exist
$d_0^+ \in L_0^+$, $d_1^+ \in L(C_1)$, and a compact arc
$\gamma_0\subset \ov\Omega_1$ joining $d_0^+$ and $d_1^+$, such that the
restriction of $X$ to the non bounded  connected subset $V_1$ of
$\Omega_1$
with boundary $\gamma_0 \cup \gamma_1 \cup \gamma_2$ a part of
$L_0^+ \cup L_0^-$ and a compact part of $L(C_1)\cup L(-C_1)$,
is an embedding. By construction,
$\ov \Omega_1 \setminus V_1$
is a compact part of $\ov \Omega_1$.}
\end{proof}

\section{Complete minimal surfaces in $\hi2\times \R$ with finite total
curvature}\label{Sec.Surface}

\hspace{.1cm} The aim of this section is to prove
the Main Theorem stated in the Introduction. The proof makes essential use of
the geometric properties of the horizontal sections of a finite
total curvature end, that were established in Section \ref{Sec.End}.

In the following, $\hi2 \times \{0\}$ will be identified with  $\hi2.$

\begin{definition}
 Let $X:=(F,h) : U \rightarrow E\subset \hi2 \times \R$ be a conformal and
complete minimal annular end, where $U:=\{|z| >1\}$,
 and let $\gamma \subset \hi2$ be a geodesic.

 We say that the end $X(U)\subset   \hi2 \times \R$ is
{\em asymptotic} to the vertical geodesic plane $\gamma \times \R$ if,
{   for any
real number $C$ with $|C|$ large enough, $E\cap \{ t=C\}$ is a complete
curve of $\hi2 \times \{C\}$} and if, for any
$\varepsilon >0$,  there exists a compact subset  $K\subset U$ such that
the distance between any point of $X(U\setminus K)$ and  $\gamma \times \R$ is
smaller than $\varepsilon$.
\end{definition}

\begin{lemma}\label{L.degre zero}
{   Let  $X:=(F,h) : U \rightarrow \hi2 \times \R$ be a conformal and complete
minimal annular end asymptotic to a vertical geodesic plane. Let
$m\in \n$ be the degree of the end $E:= X(U)$ with respect to the
parametrization $X$ $($see Definition $\ref{D.degre})$.

Then $E$ is embedded (up to a compact part). Furthermore, up to a compact part,
there exists a covering map $\pi : U\rightarrow U$ with degree
$m+1$, and a conformal minimal immersion $Y:U\rightarrow \hi2 \times \R$
such that:
\begin{itemize}
\item $X=Y \circ \pi$,

\item  $Y$ is an embedding,

\item the degree of the end $E$
with respect to the parametrisation $Y$ is 0.
\end{itemize}
Therefore, up to choose a new parametrization, we can assume that such an end
has degree zero}
\end{lemma}

\begin{proof}
{   We consider the
open sets $\Omega_k\subset U$,  $k=0,\cdots,2m+1$, as defined in
\eqref{Eq.Open}. With the aid of Lemma \ref{L.plonge}, up to remove a compact
part of $U$, we may assume that the restriction of $X$ to each $\Omega_k$ is an
embedding.

\medskip

On one hand, we know that there exists $C_1 >0$ such that for any $C>C_1$
the level set
$\{ h(z)=C\}$ is composed of $(m+1)$ complete curves
$L_0 (C),\dots, L_{m+1}(C)$ with $L_j (C)\subset \Omega_{2j}$,
$j=0,\dots,m$, see Lemma \ref{L.Level set}.

\medskip

On the other hand, since the end $E$ is asymptotic to a vertical geodesic
plane, there exists $C_2>0$ such that for any $C>C_2$ the intersection
$E\cap \{ t=C\}$ is composed of a complete curve.

\medskip

Consequently, for any $C > C_1+C_2$ we have that
\begin{equation*}
 X\big( L_0(C)\big)= X\big( L_1(C)\big)=\cdots = X\big( L_{m+1}(C)\big).
\end{equation*}

 By making vary $C$ in $]C_1+C_2, +\infty [$, we obtain
 that $X(\Omega_0), X(\Omega_2),\dots, X(\Omega_{2m})$ agree on
an open set. We deduce with the analytic continuation principle that, up to a
compact part, we have $X(\Omega_0)=X(\Omega_2)=\cdots = X(\Omega_{2m})$

\medskip

For analogous reasons, since we have also $L_j(C)\subset \Omega_{2j+1}$,
$j=0,\dots,m$, we obtain that
$X(\Omega_1)=X(\Omega_3)=\cdots = X(\Omega_{2m+1})$ up to a compact part.

Thus, up to remove a compact part of $U$ and $E$, we can assume that
$X : U \rightarrow E$ is a covering map with degree $m+1$.

For any $z_1,z_2 \in U$ we set $z_1 \sim z_2$ if
$X(z_1)=X(z_2)$. Then the canonical projection
$\pi : U \rightarrow U/\!\!\!\sim$, is a covering map with degree $m+1$.
For any $p\in U/\!\!\!\sim$, we set $Y(p)=X(z)$ for any $z\in U$ verifying
$\pi (z)=p$, by construction $Y(p)$ does not depend on the choice of such
a $z$.

Observe that $U/\!\!\sim$ is homeomorphic to an annulus. Since
$Y$ is a conformal and minimal immersion with finite total curvature, we deduce
that $U/\!\!\sim$ is conformally equivalent to $U$. Thus we may assume that
$U/\!\!\sim \,=U$ and $Y: U\rightarrow E\subset \hi2 \times \R$ is a complete
and
minimal immersion with finite total curvature. We deduce from  Lemma
\ref{L.plonge} that $Y$ is an  embedding.

\smallskip

By construction, for any $C>0$ large enough, $Y^{-1} (\{t=C\})$ is
composed of a unique and complete curve, namely
$\pi (L_k(C))$, for any $k\in \{0,\dots,m+1\}$. Consequently, if
$n \in \n$ denotes the degree of the end $E$ with respect to the
parametrization $Y$, since $Y^{-1} (\{t=C\})$ is
composed of $n+1$ complete and disjoint curves, we deduce that  $n=0$.
Thus the degree of the end $E$ with respect to the parametrization $Y$ is zero,
this completes the proof. }
\end{proof}

\medskip

\begin{definition}\label{D.horizontal graph}
Let $\gamma \subset \hi2$ be a geodesic. We say that a nonempty set
$S \subset \hi2 \times \R$ is a {\em horizontal graph with respect to
the geodesic
$\gamma$}, if for any equidistant line
$\wt \gamma$ of $\gamma$ and for any $t \in \R$, the curve $\wt \gamma \times
\{t\}$
intersects $S$ at most at one point.\end{definition}

\begin{remark}
We notice that a different notion of horizontal graph appears in \cite{SaEarp}, in order  to treat    different kinds of problems
about minimal surfaces in $\hi2\times{\mathbb R}.$
\end{remark}

\begin{prop}\label{P.non existence}
 Let $\gamma_1$ and $\gamma_2$ be two distintc geodesics in $\hi2$ with a
common asymptotic
point.  Then, there is no complete, connected, immersed minimal surface with
finite
total
curvature
and two ends, one being asymptotic to $\gamma_1 \times \R$ and the other
asymptotic to   $\gamma_2 \times \R$.
\end{prop}

\begin{proof}
(see Figure \ref{Fig4}).
 We set $\pain\gamma_1=\{a_\infty, p_\infty\}$ and
$\pain\gamma_2=\{b_\infty, p_\infty\}$, so that $p_\infty$ is the common
asymptotic point of $\gamma_1$ and $\gamma_2$. We denote by $\gamma_\bot$ the
geodesic   such that $\gamma_1$ is the
reflection of $\gamma_2$ across $\gamma_\bot$, we then have
$p_\infty \in \pain \gamma_\bot$. We denote by $\gamma_0$  the geodesic
such
that $\pain \gamma_0=\{a_\infty, b_\infty\}$. Observe that $\gamma_0$ meets
$\gamma_\bot$ orthogonally at some point $p_0\in \gamma_0\cap \gamma_\bot$.
For any $s>0$ we denote by $p_s$ the point in the half geodesic
$[p_0, p_\infty\mathclose[\subset \gamma_\bot$ such that $d_{\hi2}(p_0,p_s)=s$.

 For any $s>0$ let $\gamma_s$ be the geodesic
orthogonal to $\gamma_\bot$  at $p_s$. We set $P_s := \gamma_s \times \R$.

\begin{figure}[!ht]
\centerline{\includegraphics[scale=0.4]{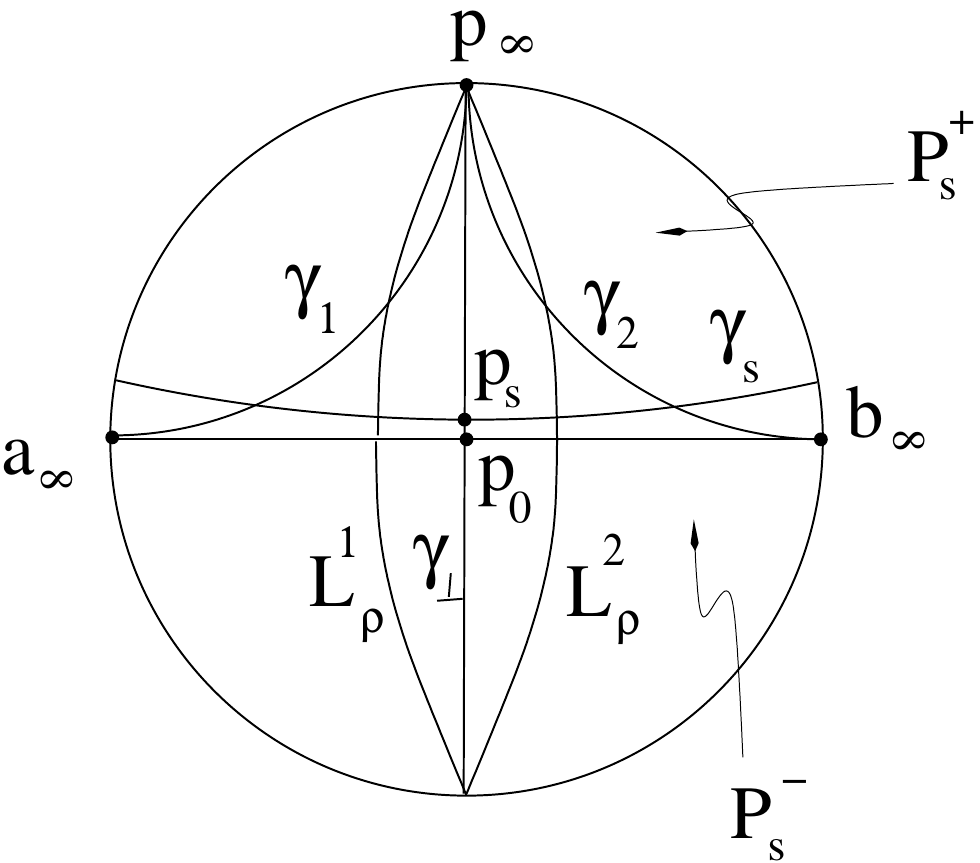}}
\caption{}\label{Fig4}
\end{figure}

Assume by contradiction that there exists a  complete and connected minimal
surface
$\Sigma$
with finite total curvature
and two ends, one  asymptotic to $\gamma_1 \times \R$ and the other
 asymptotic to   $\gamma_2 \times \R$. By a result from A. Huber
\cite[Theorems 13 and 15]{Hub}, such  a surface is
parametrized by a Riemann surface $M$ conformally equivalent to a compact
Riemann surface $\ov M$ punctured at two points $z_1$, $z_2$,
$M \simeq \ov M \setminus \{z_1,z_2\}$. We denote by
$X=(F,h) : M \rightarrow \Sigma\subset \hi2 \times \R$ the minimal and
conformal immersion. Thus, $F : M \rightarrow \hi2$ is a harmonic map
and $h : M \rightarrow \R$ is a harmonic function.

\smallskip

We want to show that $\Sigma$ is a horizontal graph with respect to
$\gamma_\bot$, afterwards we will derive a contradiction to conclude that
such a surface does not exist.

\smallskip

 For any $s>0$, we denote by $P_s^-$ the component of
$(\hi2\times \R) \setminus P_s$ containing $\{p_0\} \times \R$ and we denote by
$P_s^+$ the other component. Thus $\{p_\infty\}\times \R \subset \pain P_s^+$.
For any
$s>0$ we set
$\Sigma_s^- := \Sigma \cap P_s^-$,
$\Sigma_s^+ := \Sigma \cap P_s^+$ and we denote by
$\Sigma_s^{-*}$ the reflection of $\Sigma_s^-$ across
the vertical geodesic plane $P_s$.

For any $\rho >0$, we denote by $L_\rho^1$ (resp.  $L_\rho^2$) the equidistant
line
of $\gamma_\bot$ with distance $\rho$, intersecting $\gamma_1$ (resp.
$\gamma_2$).
We denote by $\mathcal{C}_\rho$ the domain of $\hi2 \times \R $
  bounded by $ (L_\rho^1 \cup L_\rho^2 )\times \R  $ and we set
$\mathcal{Q}_\rho := (\hi2 \times \R) \setminus \mathcal{C}_\rho$. Thus,
$\{a_\infty\}\times\R, \{b_\infty\}\times\R \subset \pain\mathcal{Q}_\rho  $ for
any
 $\rho >0$. Since each end of $\Sigma$ is asymptotic to one of the vertical
planes
$\gamma_i \times \R$, $i=1,2$, for any $s>0$ there exists $\rho_s>0,$ large
enough, such that
$\Sigma \cap \mathcal{Q}_{\rho_s} \subset P_s^-$.

{   Following Lemma \ref{L.degre zero} it can be assumed that
each end has degree zero with respect to a suitable parametrization.
Therefore, we deduce from Lemma \ref{L.Level set} that
for $s>0$ small enough, for each end $E$  and for any real number
$t$, the level set
$E \cap (\hi2 \times\{t\}) \cap P_s^-$ has only one non bounded component.
As a consequence of Propositions \ref{P.courbure} and \ref{P.graph},
we  have that if $s>0$ is
small enough, then for any $t\in \R$,  the level set
$\Sigma_{s}^- \cap (\hi2 \times\{t\})$
is a horizontal graph
with respect to the geodesic $\gamma_\bot$. } Consequently, there exists $s_1>0$
such that $\Sigma_{s_1}^-$ is a horizontal graph with respect to $\gamma_\bot$
and
$ \Sigma_s^{-*} \cap \Sigma_s^+ =\emptyset  $ for any $0< s < s_1$.

\

We set
\begin{equation*}
 I:=\{ \sigma \geq  0 \mid \Sigma_s^{-*} \cap \Sigma_s^+ =\emptyset,\
\text{for any } 0<s \leq \sigma \}.
\end{equation*}

 In order to ensure   that $\Sigma$ is a horizontal graph, we must show that $I=[0,+\infty\mathclose[.$

The set $I$ is nonempty because  $[0,s_1]\subset I.$

We set $s_2 =\sup I$. If $s_2=+\infty$ we are done. Assume that
$s_2\not=+\infty$.
By a continuity argument we have
$\Sigma_{s_2}^{-*} \cap \Sigma_{s_2}^+ =\emptyset$, so that
$s_2 \in I$.

Recall that one end of $\Sigma$ is asymptotic to $\gamma_1 \times \R$ and the
other end is asymptotic to  $\gamma_2 \times \R$. Moreover,
 from Lemma \ref{L.Level set}, formula (\ref{Eq.Ima}) and
Proposition \ref{P.courbure},  it follows  that,
for any $\varepsilon >0$, there exists
$t_0 >0$ such that for any $t>t_0$, the intersection
$\Sigma \cap (\hi2 \times \{t\})$ is composed of two complete curves, $c^t_1$
and    $c^t_2$ verifying
\begin{equation*}
 \pain c^t_1 =\{a_\infty, p_\infty\},\  \pain c^t_2 =\{b_\infty,
p_\infty\}\quad  \text{and}\quad \sup_{c^t_i}\rvert \kappa (q)\rvert <
\varepsilon,\
i=1,2,
\end{equation*}
where $\kappa$ denotes  the geodesic curvature.
From Proposition \ref{P.C1} we deduce that  that $c^t_i$ is $C^1$-close to $\gamma_i$,
$i=1,2$, if $\varepsilon$ is small enough. Analogously  $\Sigma \cap (\hi2
\times
\{-t\})$ is composed of two
complete curves, $c^{-t}_1$ and    $c^{-t}_2$, $C^1$-close to $\gamma_1$ and
$\gamma_2$ respectively.

\

 {\em {\bf Claim 1.}  There exist $t_0 >0$ and $\eta_1>0$ such that
$\Sigma_{s_2+ \eta_1}^- \cap (\hi2 \times \{\rvert t\rvert > t_0\})$ is a
horizontal graph with respect to the geodesic $\gamma_\bot$ and
 $\big( \Sigma_{s_2+ \eta_1}^{-*} \cap \Sigma_{s_2+ \eta_1}^+\big) \cap
 (\hi2 \times \{\rvert t\rvert > t_0\})= \emptyset$.}

\smallskip

We set $p_1^t=c_1^t \cap P_{s_2}$. Observe that the (nonoriented) angle between
$c_1^t$ and the equidistant line to $\gamma_\bot$ passing through $p_1^t$
is
close to the angle between the same equidistant line and $\gamma_1$ at the
common point.
Hence, there is $\alpha\in(0,\pi/2)$ such that, if $t_0$ is large enough, this
angle is  larger than
$\alpha$ for any $t>t_0$, and the same is true for  the analogous angles defined
 for  $P_{s_2}\cap c_2^t$,
$P_{s_2}\cap c_1^{-t}$ and $P_{s_2}\cap c_2^{-t}$. Therefore, there exists
$\eta_1 >0$ such that
$\Sigma_{s_2+ \eta_1}^- \cap (\hi2 \times \{\rvert t\rvert > t_0\})$ is a
horizontal graph with respect to $\gamma_\bot$ and
 $\big( \Sigma_{s_2+ \eta_1}^{-*} \cap \Sigma_{s_2+ \eta_1}^+\big) \cap
 (\hi2 \times \{\rvert t\rvert > t_0\})= \emptyset.$ Then the claim is proved.

\

 {\em {\bf Claim 2.} There exists $\eta_2>0$ such that
$\Sigma_{s_2+ \eta_2}^- \cap (\hi2 \times \{\rvert t\rvert \leq  t_0\})$ is a
horizontal graph with respect to the geodesic $\gamma_\bot$ and
 $\big( \Sigma_{s_2+ \eta_2}^{-*} \cap \Sigma_{s_2+ \eta_2}^+\big) \cap
 (\hi2 \times \{\rvert t\rvert \leq t_0\})= \emptyset$.}

\smallskip

 Observe first that,
at any  point of $\Sigma\, \cap P_{s_2}$,
the equidistant line to $\gamma_\bot$ passing through this point is not tangent
to
$\Sigma$.
Indeed, suppose that at some point $p\in \Sigma \cap  P_{s_2}$ the equidistant
line to $\gamma_\bot$ passing through $p$ is tangent to $\Sigma$. Thus $\Sigma$
is orthogonal to $P_{s_2}$ at $p$ and, therefore,
$\Sigma_{s_2}^{-*} $ and $\Sigma_{s_2}^+$ are tangent at the point $p$ of their
common boundary.
Since  $\Sigma_{s_2}^{-*} \cap \Sigma_{s_2}^+ =\emptyset$,
the boundary maximum
principle would imply that  $\Sigma_{s_2}^{-*} = \Sigma_{s_2}^+$.
This gives a contradiction,
since the asymptotic boundary of $\Sigma$ is not symmetric with
respect to any vertical geodesic plane $P_s$.

Therefore, since  $(\Sigma \cap  P_{s_2}) \cap \{   \rvert t \rvert
\leq t_0\}$ is
compact, there  is $\beta\in(0,\pi/2)$ such that the nonoriented angle
between the equidistant lines to $\gamma_\bot$ and $\Sigma$ at any point of
$(\Sigma \cap P_{s_2}) \cap \{  \rvert t \rvert \leq t_0\}$ is larger
than $\beta$. By a compactness argument again, there is $\eta_2>0$ such that
$\Sigma_{s_2+ \eta_2}^- \cap (\hi2 \times \{\rvert t\rvert \leq  t_0\})$ is a
horizontal graph and
 $\big( \Sigma_{s_2+ \eta_2}^{-*} \cap \Sigma_{s_2+ \eta_2}^+\big) \cap
 (\hi2 \times \{\rvert t\rvert \leq t_0\})= \emptyset.$ This proves the claim.

\

We set
 $\eta =\min\{\eta_1,\eta_2\}$. From  Claims 1 and 2, we get that $s_2+\eta \in
I$.
This gives a
contradiction with the maximality of $s_2$. Therefore, $I=[0,+\infty\mathclose[$
and $\Sigma$ is a horizontal graph with respect to $\gamma_\bot$.

 \smallskip

 Now we can conclude the proof.

 \smallskip

Let $Q(F)$ be the quadratic Hopf differential associated to $F$. We know
that  $Q(F)$ is holomorphic on $M$ and has a pole at the ends
$z_1, z_2 \in \ov M$. {   Let us denote by $m_1, m_2 \in \n$ the degrees of
the ends of $\Sigma$ with respect to the parametrization $X$.
 Therefore,  one end
  is a pole of order $2m_1+4 $ of $Q(F)$ and the other end is a pole of order
 $2m_2+4 $ of $Q(F)$.
According to  the Riemann relation for  $Q(F),$ we have  that
\begin{equation*}
\text{Pole}(Q(F)) -\text{Zero} (Q(F))=2\chi (\ov M),
\end{equation*}
thus $ \text{Zero} (Q(F))=\text{Pole}(Q(F))-2\chi (\ov M)=
2(m_1+m_2)+ 8-2\chi (\ov M)\geq 4$. }
Consequently, there exists $z_0\in M$ which is a zero of $Q(F)$.
Since $Q(F)=\phi (z) dz^2$, we deduce from
(\ref{F.metrique}) that $z_0$  is a pole of
$\omega$ and
then, the tangent plane of $\Sigma$ at $X(z_0)$
is horizontal (see formula \eqref{F.n3}).

 Let
$s^\prime >0$ such that $X(z_0)\in P_{s^\prime}$. We get a contradiction
by the boundary maximum principle since, on one hand
  $ \Sigma_{s^\prime}^{-*} \cap \Sigma_{s^\prime}^+ = \emptyset$
(since $\Sigma$ is a horizontal graph with respect to $\gamma_\bot$)
 but
on the other hand $ \Sigma_{s^\prime}^{-*}$ and $\Sigma_{s^\prime}^+$ are
tangent at their
common boundary point $X(z_0)$.
\end{proof}

\begin{prop}\label{P.non existence 2}
{
 Let $\gamma_1$ and $\gamma_2$ be two distinct geodesics  in
$\hi2$ intersecting at some point.  Let $\Sigma\subset\hi2 \times \R$ be
a complete immersed minimal surface with finite
total
curvature
and two ends, one being asymptotic to $\gamma_1 \times \R$ and the other
asymptotic to   $\gamma_2 \times \R$.

Then we have $\Sigma=(\gamma_1 \times \R) \cup (\gamma_2 \times \R)$ and,
consequently, $\Sigma$ has zero total curvature.}
\end{prop}

\begin{proof}
 We set $\{w\}:=\gamma_1 \cap \gamma_2$. We denote by $\alpha$ and $\beta$
the two geodesics
passing through $w$ such that the reflection of $\gamma_1$
across $\alpha$ is $\gamma_2$ and the reflection of $\gamma_1$
across
$\beta$ is $\gamma_2$. Thus,  $\alpha$ intersects  $\beta$ orthogonally at
$w$ (see Figure \ref{Fig.Prop3.2}).

\begin{figure}[!ht]
\centerline{\includegraphics[scale=0.4]{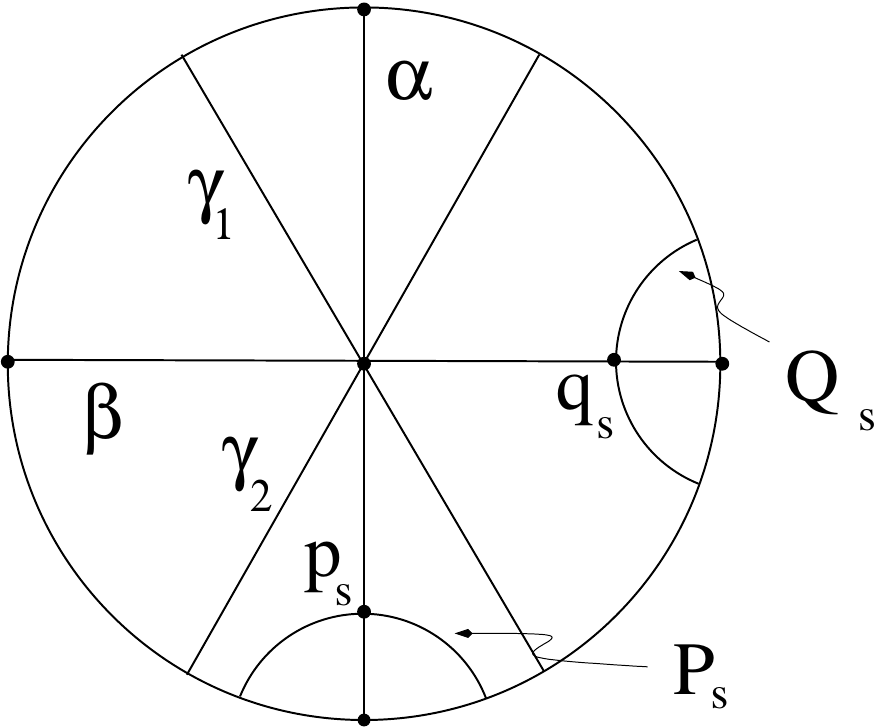}}
\caption{$w=\gamma_1\cap\gamma_2=0.$}\label{Fig.Prop3.2}
\end{figure}

 We choose an orientation on $\alpha$ and $\beta$. For any $s\in \R,$ we denote
by $p_s$ (resp. $q_s$) the point of  $\alpha$ (resp. $\beta$) whose signed
distance to $w$ is $s$, observe that $q_0=p_0=w$. Furthermore, for any $s\in \R,$
we denote by
$P_s$  (resp. $Q_s$) the vertical geodesic plane passing through $p_s$
(resp. $q_s$) and orthogonal to the geodesic $\alpha$ (resp. $\beta$), note
that $P_0=\beta \times \R$ and $Q_0=\alpha \times \R$.

We set
$ {P_0^+}:= \cup_{s> 0}\, P_s$,  $ {P_0^-}:= \cup_{s>0}\, P_s$,
$ {Q_0^+}:= \cup_{s> 0}\, Q_s$ and
${Q_0^-}:= \cup_{s> 0}\, Q_s$.

Assume that there exists a  complete  minimal
surface $\Sigma$ with finite total curvature
and two ends, one being asymptotic to $\gamma_1 \times \R$ and the other
being asymptotic to   $\gamma_2 \times \R$.

Using the Alexandrov reflection principle with respect to the vertical planes
$P_s$, $s\in \R$, we can show, as in the proof of Proposition
\ref{P.non existence}, {   that $\Sigma$ is symmetric with respect to $P_0$,
and that $\Sigma \cap \ov {P_0^+}$
is a horizontal graph
with respect to the geodesic $\alpha$ and so is
$\Sigma \cap \ov {P_0^-}$. In the same way, we can show that
$\Sigma$ is symmetric with respect to $Q_0$, and that
$\Sigma \cap \ov {Q_0^+}$ and $\Sigma \cap \ov {Q_0^-}$ are both horizontal
graphs  with respect to the geodesic $\beta$.

We deduce that $\Sigma$ is transversal to both $P_0$ and $Q_0$. Therefore the
intersections $\Sigma \cap P_0$ and $\Sigma \cap Q_0$ are analytic sets}.

\smallskip

Now we proceed as in the proof of \cite[Theorem 3, Case 1]{Schoen}.
\smallskip

We set $L:= P_0 \cap Q_0=\{w\} \times \R$. Since $\Sigma \cap \ov {P_0^+}$
and  $\Sigma \cap \ov {P_0^-}$ are horizontal graphs with respect to $\alpha$,
the self intersection set $S$ of $\Sigma$ is contained in $P_0$. By the same
argument,
we have $S\subset Q_0$, so that $S\subset L$. Since an end of $\Sigma$ is
asymptotic to $\gamma_1 \times \R$ and the other end is asymptotic to
$\gamma_2 \times \R$, we have $S\not= \emptyset$. By the analyticity of the
sets
$\Sigma \cap P_0$ and $\Sigma \cap Q_0$,
we get that
$S=L$. Moreover, since $\Sigma \cap \ov {P_0^+}$
and  $\Sigma \cap \ov {P_0^-}$ are horizontal graphs with respect to $\alpha$,
we deduce that $\Sigma \cap Q_0 =L$. Analogously, $\Sigma \cap P_0 =L$.
Therefore,  $\Sigma \setminus L$ consists of four connected components
$\Sigma_i,\ i=1,\cdots,4$ with:
\begin{equation*}
 \Sigma_1 \subset P_0^+ \cap Q_0^+ ,\quad \Sigma_2 \subset P_0^+ \cap Q_0^-,
\quad  \Sigma_3 \subset P_0^- \cap Q_0^- \quad \text{and} \quad
\Sigma_4 \subset P_0^- \cap Q_0^+.
\end{equation*}
Denoting by $\sigma$ the rotation
about the vertical geodesic $L$ with angle $\pi$, the reflection principle shows
that  $\sigma (\Sigma_1)=\Sigma_3$, so that
$\Sigma^\prime := \Sigma_1 \cup \Sigma_3 \cup L$ is a smooth and complete
minimal
surface embedded in $\hi2 \times \R$. {    Up to a change of numbering, we
can assume that $ \Sigma^\prime$ is asymptotic to $\gamma_1 \times \R$.
Therefore it can be shown using the Alexandrov reflection principle that
$ \Sigma^\prime=\gamma_1 \times \R$.

In the same way, it can be shown that
$\Sigma_2 \cup \Sigma_4\cup L=\gamma_2 \times \R$.
Therefore,  we get that
$\Sigma = (\gamma_1 \times \R )\cup (\gamma_2 \times \R)$,
which concludes the
proof.}
\end{proof}

 Now we  can restate the Main Theorem, announced in the Introduction, in the following way.

\begin{theo}\label{T.main theorem}
Let $\Sigma$ be a complete, connected
 minimal
surface immersed in $\hi2 \times \R$ with finite {   nonzero}
total curvature and two ends. Assume that each end is asymptotic
to a vertical geodesic plane $\gamma_i \times \R$,
where each $\gamma_i$, $i=1,2$, is a geodesic.

 Then, we have $ \gamma_1 \cap \gamma_2 =\emptyset$,
$\pain \gamma_1 \cap \pain\gamma_2 =\emptyset$. Furthermore, $\Sigma$
 is a properly
embedded annulus and is a horizontal catenoid.
\end{theo}

\

In order to prove Theorem \ref{T.main theorem} we fix some notations and
prove some lemmas.

\

 {\bf Notations.}
 Let  $\gamma_1, \gamma_2 \subset \hi2$ be two geodesics satisfying
\begin{equation}\label{geodesiques disjointes}
  \gamma_1 \cap \gamma_2 =\emptyset \quad \textrm{and} \quad
\pain \gamma_1 \cap \pain\gamma_2 =\emptyset.
\end{equation}
We denote by $\gamma_0 \subset \hi2$ the geodesic   orthogonal
to both $\gamma_1$ and $\gamma_2$. We set $p_1=\gamma_1 \cap \gamma_0$ and
 $p_2=\gamma_2 \cap \gamma_0$.
We call $p_0$  the middle point of the  segment of  $ \gamma_0$  between $p_1$ and $p_2.$
We denote by $\Gamma\subset \hi2$  the geodesic   passing through $p_0$ and
orthogonal to $\gamma_0$ (see Figure~\ref{Fig.Notations}).

\begin{figure}[!ht]
\centerline{\includegraphics[scale=0.4]{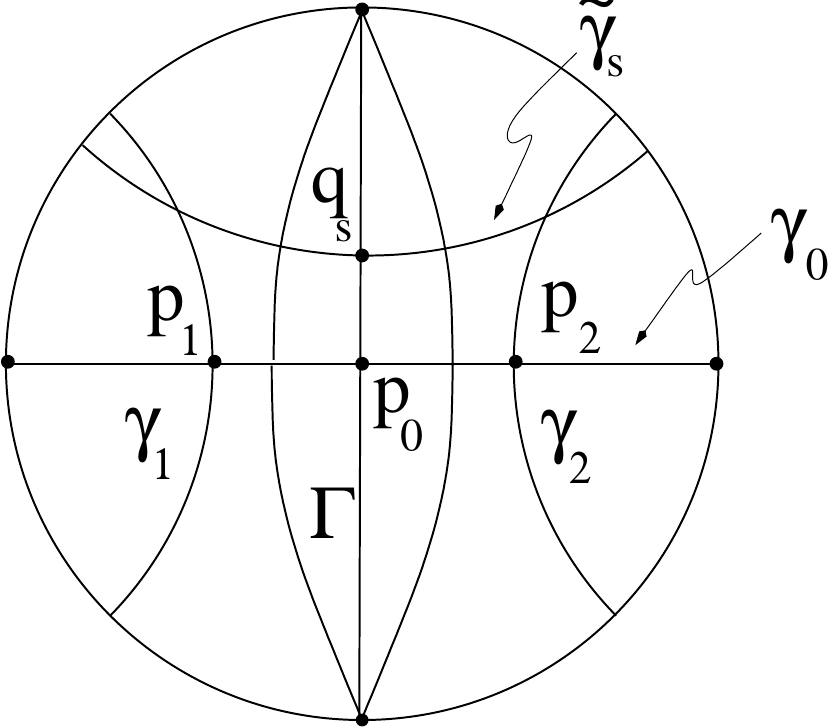}}
\caption{}\label{Fig.Notations}
\end{figure}

\bigskip

In the following lemmas the surface $\Sigma$  satisfies the hypothesis of
Theorem \ref{T.main theorem}.

\begin{lemma}\label{L.symetrie-1}
Suppose that  $\gamma_1$ and $\gamma_2 $  are geodesics satisfying  the properties
in
$(\ref{geodesiques disjointes})$. Then,
the surface
$\Sigma$ is symmetric with respect to the
vertical geodesic plane $\gamma_0 \times \R$ and
{   the closure of}
 each component of $\Sigma\setminus (\gamma_0 \times \R) $
is a horizontal graph with respect to
$\Gamma$.
\end{lemma}

\begin{proof}
  We choose an orientation on $\Gamma$.
For any $s\in \R$, we denote by $q_s$ the unique point in $\Gamma$
whose
signed distance to $p_0$ is $s$, thus $q_0=p_0$. For any $s\in \R,$ we denote
by
$\wt \gamma_s \subset \hi2$ the geodesic   orthogonal to $\Gamma$ and
passing
through $q_s.$ Observe that $\wt \gamma_0=\gamma_0$.

 For any $s\in \R,$ we set $Q_s:= \wt \gamma_s \times \R$. Moreover,
for  any $s\not=0,$ we denote by $Q_s^+$ the component of
$(\hi2\times \R) \setminus Q_s$ containing $\{p_0\} \times \R$ and  by
$Q_s^-$ the other component.  For any $s\not=0$ we set
$\Sigma_s^- := \Sigma \cap Q_s^-$,
$\Sigma_s^+ := \Sigma \cap Q_s^+$ and we denote by
$\Sigma_s^{-*}$ the reflection of $\Sigma_s^-$ across
the vertical geodesic plane $Q_s$.

As in the proof of Proposition \ref{P.non existence}
(Claim 1 and Claim 2),
we can show that for any $s \not=0$,
$\Sigma_s^-$ is a horizontal graph with respect to $\Gamma$ and that
$\Sigma_s^{-*}\cap\Sigma_s^+ =\emptyset$.

 Then, passing to the limit for
$s\to 0$ from both sides, we conclude  that
$\Sigma$ is symmetric with respect to
$\gamma_0 \times \R$  and that
 each component of $\Sigma\setminus (\gamma_0
\times \R) $
is a horizontal graph with respect to $\Gamma.$
\end{proof}

\begin{remark}
It follows from the proof of Lemma \ref{L.symetrie-1} that the tangent plane at
any point of $\Sigma\setminus (\gamma_0 \times \R) $ is never horizontal.
\end{remark}

\begin{lemma}\label{L.symetrie-2}
Suppose that  $\gamma_1$ and $\gamma_2 $  are geodesics satisfying  the properties
in
$(\ref{geodesiques disjointes})$. Then,
the surface  $\Sigma$ is symmetric with respect to the
vertical geodesic plane $\Gamma \times \R$  and
{   the closure of}
 each component of $\Sigma\setminus (\Gamma \times \R) $
is a horizontal graph with respect to
$\gamma_0$. {   Furthermore $\Sigma$ is embedded}.
\end{lemma}

\begin{proof}
 Let $d>0$ be the distance between $p_0$ and $p_1$, thus we have
$d=d_{\hi2} (p_0,p_1)=d_{\hi2} (p_0,p_2)$. For any $s \in [0,d]$ we denote by
 $\wt p_s \in \gamma_0$ the unique point
between $p_0$ and $p_1,$ whose distance to $p_0$ is $s.$  Thus $\wt p_0=p_0$ and
 $\wt p_d=p_1$.

 \smallskip

We denote by $\Gamma_s\subset \hi2$ the geodesic   orthogonal to $\gamma_0$
and
passing through $\wt p_s$, thus $\Gamma_d=\gamma_1$. We set
 $P_s:= \Gamma_s \times \R$.
For any $s \in [0,d \mathclose[$ we denote by $P_s^-$ the connected component
of
$(\hi2 \times \R)\setminus P_s$ containing $\{p_1\} \times \R$ and by
 $P_s^+$ the other component. We set $\Sigma_s^-:= \Sigma \cap P_s^-$ and
$\Sigma_s^+:= \Sigma \cap P_s^+$. Furthermore, $\Sigma_s^{-*}$ denotes the
reflection of
$\Sigma_s^-$ across $P_s$. We give to the geodesic $\gamma_0$ the
orientation going from
$p_1$  to  $p_2$.

We will say that $\Sigma_s^{-*} \leq\Sigma_s^+$ if
$\Sigma_s^{-*}$ remains under $\Sigma_s^+$ with respect to the orientation
of $\gamma_0$.

{   As in the proof of Proposition \ref{P.non existence},}
it can be shown that there
exists $\varepsilon >0$ such that for any $s\in [d-\varepsilon, d\mathopen[\,$,
$\Sigma_s^-$ is a horizontal graph with respect to $\gamma_0$. Therefore,
for any $s\in [d-\varepsilon/2, d\mathopen[$, we have $\Sigma_s^{-*}
\leq\Sigma_s^+$.

We set
\begin{equation*}
 I=\big\{ s\in [0,d] \mid \Sigma_r^{-*} \leq\Sigma_r^+ \ \text{for any}\
r \in \mathclose]d-s, d\mathopen[ \big\}.
\end{equation*}
We have $I\not= \emptyset$, since $\varepsilon/2 \in I$. We set
$s_0:= \sup I$, we want to prove that $s_0=d$.

\smallskip

Assume that  $s_0\not=d$. By continuity we get
that $\Sigma_{s_0}^{-*} \leq\Sigma_{s_0}^+.$ On the other hand we have
$\Sigma_{s_0}^{-*} \not= \Sigma_{s_0}^+,$ since the asymptotic boundaries of
those two
parts are not equal. Observe that $\partial\Sigma_{s_0}^{-*} = \Sigma \cap
P_{s_0}$ is
compact and the boundary maximum principle shows that $\Sigma$ is
never
orthogonal
to $ P_{s_0}$ along their intersection. We deduce that there exists
$\varepsilon_1 >0$
such that $\Sigma_{d-s_0-\varepsilon_1  }^-$ is a horizontal graph with respect
to
$\gamma_0$ and $    \Sigma_r^{-*} \leq\Sigma_r^+ $ for any
$r \in \mathclose]d-s_0-\varepsilon_1, d\mathopen[  $, which gives a
contradiction with
the maximality of $s_0$. We deduce that $s_0=d$ and then $ \Sigma_0^{-*}
\leq\Sigma_0^+ $.

 Using the same arguments coming from the other side, that is from $p_2$ to
$p_0$, we can
show that $\Sigma_0^{+*} \geq\Sigma_0^- $. We conclude that $\Sigma_0^{+*}
=\Sigma_0^-$,
that is $\Sigma$ is symmetric with respect to $P_0=\Gamma\times \R$,
as desired.

{   The proof that $\Sigma$ is embedded can be established in the same way as
in \cite[Theorem 2]{Schoen}}.
\end{proof}

\bigskip

\begin{remark}
  In  \cite[Proposition 2.4]{MMR}, F. Martin, R. Mazzeo and M. Rodriguez
have given an independent proof of Lemmas  \ref{L.symetrie-1} and
\ref{L.symetrie-2}.
\end{remark}

\bigskip

\begin{lemma}\label{L.symetrie-3}
Suppose that  $\gamma_1$ and $\gamma_2 $  are geodesic satisfying the properties
in
$(\ref{geodesiques disjointes})$. Then, the surface $\Sigma$ is symmetric
with respect
to
some slice $\hi2 \times \{t_0\}$ and {   the closure} each component of
$\Sigma \setminus (\hi2 \times \{t_0\})$ is a vertical graph.
\end{lemma}

\begin{proof}
 {   From Lemma \ref{L.symetrie-2}
we know that:  $\Sigma$ is embedded,
 each component of
$\Sigma\setminus (\Gamma\times \R)$ is a
horizontal graph with respect to $\gamma_0$ and that
$\Sigma$ is symmetric
with respect to the geodesic vertical plane $\Gamma\times \R$.

We first deduce that
$\Sigma$ is transversal to $\Gamma\times \R$, then that $\Sigma$ is actually
orthogonal to $\Gamma\times \R$. Therefore the intersection
$\mathcal{C}:=\Sigma \cap (\Gamma\times \R)$ is composed of a
finite number of Jordan
curves. Since  each component of
$\Sigma\setminus (\gamma_0\times \R)$ is a
horizontal graph with respect to $\Gamma$ (see Lemma \ref{L.symetrie-1}), we
get that the interiors of the Jordan curves of $\mathcal{C}$  are
pairwise disjoint.}

\smallskip

For $i=1,2$, we call $\Sigma_i$ the component of
$\Sigma\setminus (\Gamma\times \R)$
which is asymptotic to the vertical plane $\gamma_i \times \R$. For any
$t\in \R$,
we set $\Pi_t := \hi2 \times \{t\}$. Let $t_1\in \R$ be such that
$\partial\Sigma_1 \cap \Pi_t =\emptyset$ for any $t >t_1$ and
$\partial\Sigma_1 \cap \Pi_{t_1 }\not=\emptyset$. Such a $t_1$ exists since
$\partial\Sigma_1= \mathcal{C}$ is compact.

For any $t\in \R$ we set: $\Pi_t^+:= \hi2 \times \{s \mid s>t \}$,
$\Pi_t^-:= \hi2 \times \{s \mid s< t \}$,
$\Sigma_{1,t}^+=\Sigma_1 \cap \Pi_t^+ $,
$\Sigma_{1,t}^-=\Sigma_1 \cap \Pi_t^- $ and  we denote by
$\Sigma_{1,t}^{+*}$ the reflection of $\Sigma_{1,t}^+$ across $\Pi_t$.
Moreover, $\Sigma_{1,t}^{+*} \geq \Sigma_{1,t}^-  $ means that
$\Sigma_{1,t}^{+*}$ stays above $\Sigma_{1,t}^- $.

\

 {\em {\bf Claim 1.} For any $t\geq t_1$ we have
$\Sigma_{1,t}^{+*} \geq \Sigma_{1,t}^-  $. Consequently, $\Sigma_{1,t_1}^+$ is
a
vertical graph.}

\smallskip

Indeed, for $t>t_1$ we know from Lemma \ref{L.symetrie-2}
that the intersection
$\Sigma_1 \cap \Pi_t$ is a complete curve, that is
 a horizontal graph with respect to
$\gamma_0$ and whose  asymptotic boundary
is $\pain \gamma_1 \times \{t\}$.
 For any $s\in \R$ we denote by $T_s$ the horizontal translation
along
$\gamma_0$ of signed length $s$ in the direction going from $p_2$  to  $p_1$.

Suppose that $\Sigma_{1,t}^{+*}$ does not remain above $\Sigma_{1,t}^-  $.
Then,
for $\varepsilon >0$ small enough, the translated $T_\varepsilon
(\Sigma_{1,t}^{+*})$
does not remain above $\Sigma_{1,t}^-  $. Observe that
$\pain T_\varepsilon (\Sigma_{1,t}^{+*}) \cap
\pain \Sigma_{1,t}^-=~\emptyset$.
Moreover,
 $\partial T_s (\Sigma_{1,t}^{+*}) \cap \ov{   \Sigma_{1,t}^- }=\emptyset$,
for any $s>0$, since $\Sigma_1$ is a horizontal graph with respect to
$\gamma_0$. We deduce that, for $\varepsilon >0$ small enough, the part of
$\Sigma_{1,t}^-  $ which remains above $T_\varepsilon (\Sigma_{1,t}^{+*})$ has
compact
closure. Therefore, there exists $s_1>0$ such that
$ T_s (\Sigma_{1,t}^{+*}) \cap  \Sigma_{1,t}^-  =\emptyset$ for any $s>s_1$
and
$ T_{s_1} (\Sigma_{1,t}^{+*}) \cap  \Sigma_{1,t}^-  \not=\emptyset$. This means
that
$ T_{s_1} (\Sigma_{1,t}^{+*})$  and $ \Sigma_{1,t}^- $ are tangent at some
point and
one surface remains in one side of the other, which gives a contradiction with
the
maximum principle and proves the claim.

\

 {\em {\bf Claim 2.} For any $\varepsilon >0$ small enough, we have
$\Sigma_{1,t_1 -\varepsilon}^{+*} \geq \Sigma_{1,t_1-\varepsilon}^-  $.}

\smallskip

For any $ s\in \R$, let $T_s$ be the horizontal translation
defined as in Claim 1.
Since {   the closure of} $\Sigma_1$ is a
horizontal
graph with respect to $\gamma_0$, we have
$T_s \left( \mathcal{C}\right) \cap \Sigma_1=\emptyset $ for any $s>0$.
Furthermore,
since the whole surface $\Sigma$ is symmetric with respect to $\Gamma \times \R$
we have
that $T_s \left( \mathcal{C}\right) \cap \Sigma_1=\emptyset $ for any $s\not=0$.
Let
$\mathcal{D}\subset \Gamma \times \R $ be the bounded subset with
boundary
$\mathcal{C}$.
Since
$\Sigma$ is connected, we have $T_s (\mathcal{D}) \cap \Sigma=\emptyset$ for
any
$s\not=0$.

Let $\varepsilon>0$ such that
$\mathcal{C}_{t_1 -\varepsilon}^{+*} \geq \mathcal{C}_{t_1-\varepsilon}^-$,
where $\mathcal{C}_{t_1 -\varepsilon}^{+,*}$ etc. are obviously defined.
Then, using an argument analogous to that of Claim 1,
considering translations $T_s$,
it can be shown that
$\Sigma_{1,t_1 -\varepsilon}^{+*} \geq \Sigma_{1,t_1-\varepsilon}^-  $. This
proves
Claim 2.

\

 Now we can conclude the proof.

Since $\mathcal{C}$ is composed of a finite number of Jordan curves, there
exists a
component, say $C$, and a real number $t_0 <t_1$ satisfying
$\Sigma_{1,t}^{+*} \geq \Sigma_{1,t}^-  $ for any $t>t_0$,
such that at least one of the following properties occurs:
\begin{enumerate}
 \item $C_{t_0}^{+*} \geq C_{t_0}^-  $ and  $C_{t_0}^{+*} $ is tangent to
$C_{t_0}^- $
at some interior point.

\item $C_{t_0}^{+*} \geq C_{t_0}^-  $ and  $C_{t_0}^{+*} $ and $C_{t_0}^- $ are
tangent along their common boundary.
\end{enumerate}
Recall that $\Sigma$ is orthogonal to $\Gamma\times \R $ along
$\mathcal{C}$, since
$\mathcal{C}=\Sigma \cap (\Gamma\times \R)$ and $\Sigma$ is symmetric with
respect to
$\Gamma\times \R$.

Thus, in the first case, applying the boundary maximum principle
to the
surfaces
$\Sigma_{1,t_0}^{+*}$  and $\Sigma_{1,t_0}^-  $, we conclude
that $\Sigma_{1,t_0}^{+*} = \Sigma_{1,t_0}^-  $ and then
$\Sigma_{t_0}^{+*} = \Sigma_{t_0}^-  $.

In the second case we apply
the boundary maximum principle to the surfaces
$\Sigma_{t_0}^{+*}$ and
$\Sigma_{t_0}^- $ in order to infer $\Sigma_{t_0}^{+*} = \Sigma_{t_0}^-  $.

Consequently, the surface $\Sigma$ is symmetric with respect to the horizontal
plane
$\Pi_{t_0}=\hi2 \times \{t_0\}$, as desired.
\end{proof}

\begin{remark}
It follows from the proof of Lemma \ref{L.symetrie-3} that the tangent plane at
any point of
$\Sigma\setminus (\hi2 \times \{t_0\}) $ is never
vertical.
\end{remark}

\begin{proof}[Proof of  Theorem \ref{T.main theorem}]
 The maximum principle shows that $\gamma_1 \not= \gamma_2$, since $\Sigma$ is
not a
vertical plane.
We know from
Proposition \ref{P.non existence} that
$\pain \gamma_1 \cap \pain\gamma_2 =\emptyset$  and from
Proposition \ref{P.non existence 2}  that
$ \gamma_1 \cap \gamma_2 =\emptyset$.
 Thus, the geodesics
$\gamma_1$ and $\gamma_2$ satisfy the properties (\ref{geodesiques disjointes}).
{   Therefore we deduce from Lemma  \ref{L.symetrie-2} that $\Sigma$ is
embedded.}

Furthermore, we deduce from Lemmas \ref{L.symetrie-1},
\ref{L.symetrie-2} and
 \ref{L.symetrie-3} that $\Sigma$ is symmetric with respect to the vertical
planes
$\gamma_0 \times \R$ and $\Gamma \times \R$ and also with respect to the
 slice $\Pi_0:=\hi2 \times \{0\}$ (up to a vertical translation).

We call $S_0$ the reflection across the slice $\Pi_0$, $S_\Gamma$
the reflection across
 the vertical plane $\Gamma \times \R$ and $S_{\gamma_0}$
the reflection across
 the vertical plane $\gamma_0 \times \R$.

\

For any real number $s\not=0$, we denote by $\Gamma_s$ the equidistant line
to $\Gamma$ with distance equal to $|s|$, which intersects $\gamma_0$ between
$p_0$ and $p_1$ (resp. $p_0$ and $p_2$) if $s>0$ (resp. $s<0$). We set
$\Gamma_0=\Gamma$.
For any $s\in \R$, we set $P_s:=\Gamma_s \times \R$.

\

We define $\Sigma^+:= \Sigma \cap (\hi2 \times \mathclose] 0, +\infty\mathopen[).$
Since $\Sigma^+$ is a vertical graph, the tangent plane is never vertical along
 $\Sigma^+$. Consequently  $\Sigma^+$ intersects any $P_s$
transversally. Since
each component of  $\Sigma^+ \setminus (\gamma_0\times \R)$ is a horizontal
graph with
respect to
$\Gamma$ and since $\Sigma$ is symmetric with respect to $\Pi_0$ and
$\gamma_0\times \R$,
we deduce that for
any $s\in \R$ the intersection $\Sigma \cap P_s$ consists of a Jordan curve.
 Therefore,  $\Sigma$ is homeomorphic to an annulus. Since $\Sigma$ has finite
total
curvature, we get that $\Sigma$ is conformally parametrized by
$\C^*:=\C\setminus \{0\}$.

\

 Let $X: \C^* \rightarrow \Sigma \subset \hi2 \times \R$ be a conformal
parametrization of $\Sigma$. {   Since $\Sigma$ is embedded we may assume that
$X$ is an embedding. We deduce from Lemma \ref{L.degre zero} that each end of
$\Sigma$ has degree zero.}

\

The symmetry $S_\Gamma$ corresponds to a anticonformal
diffeomorphism
$s_\Gamma : \C \cup\{\infty\} \rightarrow \C \cup\{\infty\}$ satisfying
$s_\Gamma(0)=\infty$ and $s_\Gamma(\infty)=0$. Since the set  of
fixed
points of $S_\Gamma$ in $\Sigma$ is a Jordan curve, the set of fixed points of
$s_\Gamma$ is
a circle
$c_\Gamma$. Up to a conformal change of coordinates, we can assume that
$c_\Gamma\subset\C$ is the unit circle centered at  the origin. Thus, we get
$s_\Gamma (z)=1/\ov z$ for any $z\in \C^*$.

Then, we denote by
$s_{\gamma_0} : \C \cup\{\infty\} \rightarrow \C \cup\{\infty\}$ the
anticonformal
diffeomorphism corresponding to $S_{\gamma_0}$. The  set  of fixed points of
$s_{\gamma_0}$ in
$\C^*$ is a straight line $L_\gamma$ passing to and punctured at  the origin. Up to a
rotation we
can assume that
 $L_\gamma=\{ \Re\, z=0\}$. Thus, we have $s_{\gamma_0}(z)=
-\ov z$
for any $z\in \C$.

At last, let us call $s_0 : \C \cup\{\infty\} \rightarrow \C \cup\{\infty\}$ the
anticonformal
diffeomorphism corresponding to $S_0$.  The  set of fixed points of $s_0$ in
$\C^*$ is a straight line $L$ passing and punctured at  the origin. Since we have
$(S_0 \circ S_{\gamma_0}) \circ (S_0 \circ S_{\gamma_0})= Id$ on $\Sigma,$ we
must have
$(s_0 \circ s_{\gamma_0}) \circ (s_0 \circ s_{\gamma_0})= Id$ on $\C^*$. Thus,
$L$ must be
orthogonal
to
$L_\gamma$ and we get $L=\{ \Im\, z=0\}$. Then, we have $s_0(z)= \ov z$
for any $z\in \C$.

We call $P_0^+$ the component of $(\hi2 \times \R) \setminus (\Gamma\times\R)$
containing
$\gamma_1 \times \R$ and we set $\Sigma^{++}:=\Sigma^+\cap P_0^+$. Finally, we
call
$\Sigma_0$ any of the two components of $\Sigma^{++}\setminus (\gamma_0
\times \R)$.
Thus,  we recover the whole surface $\Sigma$ by applying the
symmetries $S_0$, $S_\Gamma$ and $S_{\gamma_0}$ to
the closure of $\Sigma_0$.

We can assume that $\Sigma_0$ is parametrized by the subset
\begin{equation*}
 U_0:=\{z\in \C \mid |z| >1, \ \Re\, z<0,\ \Im\, z>0\}.
\end{equation*}
Since $\Sigma_0$ is simply connected, we can consider its conjugate $\Sigma_0^*$
which is
a minimal surface in $\hi2 \times \R$ uniquely defined up to an ambient
isometry.

From now on, for any object $x$ relative to $\Sigma_0$ we denote by
$x^*$ the corresponding object relative to the conjugate surface $\Sigma_0^*$.
Thus, $X^*:= (F^*,h^*):U_0 \rightarrow \Sigma_0^*\subset \hi2\times \R$
is a conformal parametrization of $\Sigma_0^*$.

Observe that the boundary of $\Sigma_0$ is composed of three arcs:
\begin{enumerate}
\item A semi-complete curve $b_1 \subset (\gamma_0  \times\R)$ with
boundary point $q$.
\item A compact arc $b_2 \subset (\Gamma \times \R)$ with boundary $q$ and
$\wt q$.
\item A semi-complete curve $b_3\subset \Pi_0$ with boundary  point $\wt q$.
\end{enumerate}

In order to visualize the following discussion we consider the model of the unit
disk for $\hi2$ (see Figure \ref{Fig.Conjugue}).

Up to an isometry we can assume that $q^*=0\in \hi2.$
\begin{figure}[!ht]
\centerline{\includegraphics[scale=0.4]{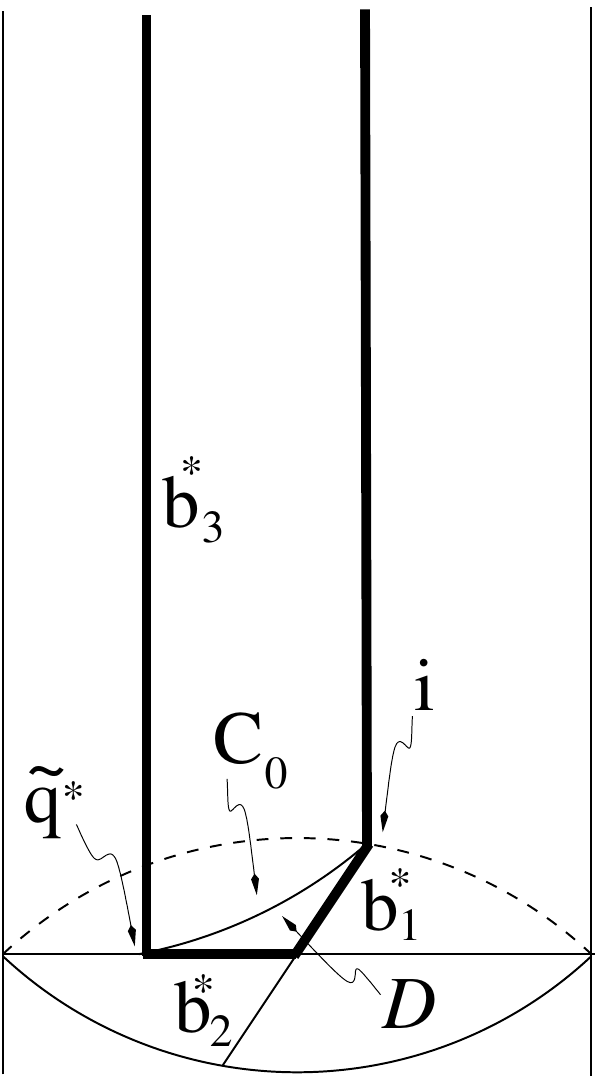}}
\caption{}\label{Fig.Conjugue}
\end{figure}

Since $b_1$ is
contained in a vertical plane, its conjugate $b_1^*$ must be a horizontal  half-geodesic
 issue
from $0$  (see \cite[end of Section 4.1]{Daniel}). We can assume that $\pain
b_1^* = \{i\}$.

Moreover, since  $b_2$ is
contained in a vertical plane, its conjugate $b_2^*$ must be
a compact geodesic arc  orthogonal to $b_1^*$ with endpoints $\wt q^*$ and $ 0.$
Finally, since  $b_3$ is
contained in a horizontal plane, its conjugate $b_3^*$ must be a vertical
half-geodesic
issue from $ \wt q^*$. We can assume that  $b_3^*= \wt q^* \times \{t\geq 0 \}$.

We denote by $C\subset \hi2$ the geodesic   passing through $ \wt q^*$,
having
$i$ in its asymptotic boundary and we denote by $C_0$
{   the half-geodesic of $C$ issue from  $ \wt q^*$ verifying
$\pain C_0 = \{i\}$}
We call
$\mathcal{D}$ the domain of $\hi2$ bounded by $b_1^*$, $b_2^*$ and $C_0$ such
that
$\pain \mathcal{D}=\{ i\}$.

\smallskip

In order to prove that $\Sigma$ is a  horizontal catenoid,
it is enough to prove that
$\Sigma_0^*$ is a vertical graph over  $\mathcal{D}$, with infinite data
on
$C_0$ and zero data on the two sides $b_1^*$ and $b_2^*$  (see  \cite{P}).

\smallskip

We call $\sigma $ the reflection in $\hi2 \times \R$ across the vertical
geodesic containing $b_3^*$.

Since $\Sigma_0^*$ is parametrized by $U_0$, the interior of
$\ov \Sigma_0^* \cup \sigma(\Sigma_0^*) $ is parametrized by
\begin{equation*}
 \wt U_0 :=\big \{z\in \C \mid |z|>1,\ \Re\, z<0 \big\}.
\end{equation*}

We recall that $\phi^*=-\phi$, $h^*=-2\,\Re\, W$, since $h=2\,\Im\, W$ and
$h^*$ is the harmonic conjugate of $h.$ Thus
$W^*=-iW$. {
We may suppose that $h^* =0$ on $b_1^* \cup b_2^*$. We are able to study
the behavior of
$F^*$ and  $h^*$ on $\wt U_0 $ in the same way as we did for $F$ and $h$ in
Section \ref{Sec.End}. Recall that $m=0$ for the end of $\Sigma$
parametrized by $\{ |z| >1\}$.

\smallskip

Since the interior of $\ov \Sigma_0\cup S_0 (\Sigma_0)$ is a horizontal graph
with respect to $\Gamma$, we get that the tangent plane along it is never
horizontal. Thus we get $\phi \not=0$ on  $\wt U_0$, and since $(h_z)^2=-\phi$,
 we get also that $h_z \not= 0$  on  $\wt U_0$. Therefore
$h$ is strictly monotonous along any level curve of $h^*$. Also, since
$h\equiv 0$ on $L^-:= \{\Im\, z=0\} \cap \wt U_0 $, we get that $h^*$ is
strictly monotonous  and unbounded along $L^-$.

We deduce that for any $\mu >0$, the level set $\{h^* = \mu\}$ is composed of
a
unique complete curve $L_\mu \subset \wt U_0$: a part of $L_\mu$ has the ray
$\{ri,\ r >0\}$ as asymptotic direction and the other part has the ray
$\{ri,\ r <0\}$ as asymptotic direction. Furthermore $L_\mu$ intersects $L^-$
at a unique point. We deduce also that for any $\mu <0$, the level set
$\{h^* = \mu\}$ is empty.}

\

Since $h$ is
strictly monotonous along any curve $L_\mu$, we deduce that $W^*$ is one-to-one
on
$\wt U_0 $. The results established in Theorem \ref{T.Description}
 yields  that the curves $F^*(L_\mu)$ in $\hi2$
converge to
the geodesic   $C$ when $\mu \to +\infty$. {   We deduce that
$\pain \Sigma_0^+ \subset\{ i\} \times \R$.}

\

Let us call $B_i$ the geodesic  containing $b_i^*$, $i=1,2$.

\

 {\em {\bf Claim 1.} We have $\Sigma_0^* \subset \mathcal{D} \times \R
$.}

\smallskip

By construction,  $B_1$ and $B_2$ meet orthogonally at  the origin and
$B_2$  is the geodesic with asymptotic points $(-1,0),$ $(1,0).$   For $s>0$, we call $Q_s \subset \hi2 \times \R$ the
vertical
plane orthogonal to  $B_2$ at $(s,0)$ and we call $Q_s^-$ the component
of $(\hi2 \times \R)\setminus Q_s$ which does not contain  $\mathcal{D} \times
\R$.

Recall that $\pain \Sigma_0^* \subset \{i\}\times \R$. Moreover, if $(z_n)$ is
a
sequence in
$U_0$ such that $h^*(z_n)\to +\infty$, then we have $d_{\hi2}\big(F^* (z_n),C
\big)\to
0$.
Consequently, for any  $s>0,$ the intersection $\Sigma_0^* \cap Q_s^-$ is either
empty or
have compact closure. Assume the latter is true.  In this case,
we find $s$ such that
$\Sigma_0^* \cap Q_s^- =\emptyset.$  Then, we start to decrease $s$. By the maximum
principle, we can decrease $s$ till $0$ and obtain that
$\Sigma_0^* \cap Q_s^- =\emptyset$ for any $s>0$. Therefore, $\Sigma_0^*$ remains in the component of
$(\hi2 \times \R)\setminus (B_1\times\R)$ containing $\mathcal{D} \times \R$.

By the same reasoning as above we can prove that  $\Sigma_0^*$ remains in
the component of
$(\hi2 \times \R)\setminus (B_2\times\R)$ containing $\mathcal{D} \times \R$
and
also  $\Sigma_0^*$ remains in the component of
$(\hi2 \times \R)\setminus (C\times\R)$ containing $\mathcal{D} \times \R$.
We conclude that  $\Sigma_0^* \subset \mathcal{D} \times \R $.

\

 {\em {\bf Claim 2.} We have $F^* (U_0)\subset \mathcal{D}$.
Furthermore,
the map $F^* : U_0 \rightarrow \mathcal{D}$ is  proper.}

\smallskip

Let $Pr : \hi2 \times \R \rightarrow \hi2$ be the projection on the first
component. Since
$F^*= Pr \circ X^*$ we deduce from Claim 1 that $F^* (U_0)\subset \mathcal{D}$.

We must prove that for any compact  set
$K\subset \mathcal{D}$,
$(F^*)^{-1} (K)$ is a compact subset of $U_0$.  In
order to prove it, it is enough to show that for any sequence
$(z_n)$ in  $(F^*)^{-1} (K)$, there is a subsequence of $(z_n)$ converging
in
$(F^*)^{-1} (K)$.

Since $K$ is far from the geodesic   $C$, the  height function is bounded on
$K$.
Therefore,  there exists a constant $\mu>0$ such that $(F^*)^{-1} (K)$
remains in the subset of $U_0$ bounded by the level set $L_\mu$ and the
half-axis
$\{iy \mid y>0 \}$.  Let $(z_n)$ be a sequence in $(F^*)^{-1} (K)$. Suppose that $(z_n)$ is not
bounded.
Therefore, there exists a subsequence
$(z_{\varphi (n)})$ of $(z_n)$
such that  $|z_{\varphi (n)}| \to +\infty$. As in the proof of
Theorem \ref{T.Description} (Assertion 2), we can show that there exists a
sequence
$(iy_n)$ such
that
$y_n \to +\infty$ and $d_{\hi2}\big(F^*(z_{\varphi (n)}), F^*(iy_n)\big)\to 0$.
But this
is {   absurd} since we have, by construction, $F^*(z_{\varphi (n)}) \in K$
and
$F^*(iy_n)\to i\in\pain \hi2$.

Thus  $(z_n)$ is a bounded sequence of $U_0$. Therefore, since $U_0\subset \C$,
we can
extract a subsequence $(z_{\psi (n)})$ converging to some point $z\in \ov U_0$.
We want
to show that $z \in (F^*)^{-1} (K) $.

Observe that $F^*$ maps the boundary of $U_0$ onto the boundary of
$\mathcal{D}$
and that $d_{\hi2}\big( K,\partial\mathcal{D}\big)>0 $, since $K$ is
a compact
subset
of $\mathcal{D}$. Therefore,  we deduce  that $z\in U_0$. Since $F^*$ is continuous
and
$K$ is
compact,  we obtain $F^*(z) \in K$, from which we get $z \in (F^*)^{-1} (K) $.
Therefore
$F^* : U_0 \rightarrow \mathcal{D}$ is  a proper map, as desired.

\

 {\em {\bf Claim 3.} We have $F^* (U_0)= \mathcal{D}$ and
$\Sigma_0^*$ is a vertical graph over $\mathcal{D}$.}

\smallskip

We know from Claim 2 that  $F^* (U_0) \subset\mathcal{D}$. Therefore,  it
suffices
to prove that $F^* (U_0)$ is a closed and open  subset of $\mathcal D.$

It is known that $n_3^*= \pm n_3$ (see \cite[Remark 9]{HST}). On the other hand,
since $\Sigma_0$
is a
vertical graph, we have $n_3 \not=0$ along $\Sigma_0$. We deduce that the
tangent plane
is never vertical along $\Sigma_0^*$ and that the map
$F^* : U_0 \rightarrow  \mathcal{D}\subset \hi2 $ is a local
smooth diffeomorphism.
Then,  $F^* $ is an open map. As $U_0$ is
open, we get that $F^* (U_0)$ is an open subset of $\mathcal{D}$.

Now we prove that $F^* (U_0)$ is also a closed subset of $\mathcal D.$

Let $(q_n)$ be a sequence in $F^* (U_0)$   converging to some point
$q\in\mathcal{D}$. We want to prove that $q\in F^* (U_0)$.

We set $K:=\{q\} \cup \{q_n,\ n \in \n\}$, then $K$ is a compact subset of
$\mathcal{D}$.
From Claim 2 we deduce that $(F^*)^{-1}(K)$ is a compact subset of $U_0$. For
any
$n\in \n$ there exists $z_n \in (F^*)^{-1}(K)$ such that
$F^*(z_n)=q_n$. Therefore, we can extract a subsequence $(z_{\varphi (n)})$
which
converges to some $z\in (F^*)^{-1}(K)\subset U_0$. Since $F^*$ is continuous we
obtain
$F^*(z_{\varphi (n)}) \to F^* (z)$, that is $q_{\varphi (n)} \to F^*(z)$. We
deduce that
$q=F^*(z)$ and then $q\in F^*(U_0)$, therefore   $F^* (U_0)$ is a closed subset of $\mathcal D.$
Consequently we get  $F^* (U_0)=\mathcal{D}$.

Hence, the map $F^* : U_0 \rightarrow  \mathcal{D}$ is
a local smooth diffeomorphism.  Moreover $F^*$ is  proper and surjective.
We deduce that it is a covering map. Since $\mathcal{D}$ is connected
and simply
connected and $U_0$ is
connected, we deduce that $F^*$ is a global diffeomorphism from $U_0$ onto
$\mathcal{D}$, that is $\Sigma_0^*$ is a vertical graph over $\mathcal{D}$,
 as desired.

\medskip

{   Since $h*$ is strictly monotonous and non bounded along $L^-$, we obtain
that $\Sigma_0^*$ is a vertical graph over $\mathcal{D}$ with infinite data on
$C_0$ and zero data on $b_1^*$ and $b_2^*$, this concludes the proof.}
\end{proof}



\section{Appendix. Basic geometry in $\hi2$}

In this section, we establish some background material about $C^2$-curves  in $\hi2$
 whose absolute value of geodesic curvature
is
strictly smaller
than one.  We observe that
the condition on the curvature implies that such a curve
is embedded, see for example
\cite[Proposition 2.6.32]{ST3}.

\begin{prop}\label{P.limit}
 Let $c:[0,+\infty\mathclose[ \rightarrow \hi2$ be a regular $C^2$-curve with
infinite
length.  Let $\kappa (t)$ be the geodesic curvature of $c$ at the point $c(t)$.
{   Assume that $\rvert \kappa(c(t)) \rvert < k <1$, for any $t\geq 0$, and
that $c$ is one-to-one.}

Then, the
curve
$C:= c([0,+\infty\mathclose[)$   has no limit
point in $\hi2,$  and the  asymptotic boundary of
$C$ consists of only one point
$\{p_\infty \}=\pain C $.
\end{prop}

\begin{proof}
$   $
If $k=0$ then $C$ is a part of a geodesic and the
assertions are obvious. Therefore we assume that $0<k<1$.

\

{\em {\bf Claim 1.}  $C$ has no limit point in $\hi2$.}

\smallskip

Indeed, assume  by contradiction that  there exists $p\in \hi2$ and a sequence of positive numbers
$(t_n)$
such that $t_n \to +\infty$ and $p_n:= c(t_n) \to p$ when $n\to \infty$.

Assume first that there  exists a point $q\in C$, $q=c(t_0)$ for
some $t_0>0$, such that
$C$ is  orthogonal  at $q$ to the geodesic   passing through $q$ and $p$. Let
$H_q\subset \hi2$ be the horocycle through $q,$ tangent to the curve $C,$ such
that  $p$ belongs to the convex component of $\hi2 \setminus H_q$.
Recall that $\rvert \kappa (c(t))\rvert <1 $ and the absolute value of
the
curvature of the horocycles is 1. Thus, the maximum principle for curves,
see \cite[Theorem 2.6.27]{ST3},
ensures that $C_0 := c([t_0, +\infty\mathclose[)$
belongs to the
non convex component of $\hi2 \setminus H_q$ and then, $p$ cannot not be in the
closure of $C$.

Hence we infer that the function $t \mapsto d_{\hi2}(c(t), p)$ is
strictly decreasing.

For $t>0$ we denote by $\alpha(t)\in[0,\pi]$, the nonoriented
angle at $c(t)$ between the tangent vector $c^\prime (t)$ and the geodesic
segment $[c(t), p]$. Since the function $t \mapsto d_{\hi2}(c(t), p)$ is
strictly decreasing we have  $\alpha(t)\in[0,\pi/2[$ for any $t>0$.

Actually, we have $\alpha (t) \rightarrow 0$ as $t \to +\infty$. Indeed, assume
by contradiction that there exist a sequence $(t_n)$, and a real number
$ \alpha_0 \in \,]0,\pi/2[$, such that $t_n\to +\infty$ and $\alpha (t_n)\to
\alpha_0$.
For any $n\in \n$, we denote  by $\gamma_n$ the geodesic of $\hi2$
through $c(t_n)$ tangent to $C$.
Let us denote by $H_{n}\subset\hi2$, $n\in \n$, the horocycle through $c(t_n)$
tangent to the curve $C$ and contained in the same  component
of
$\hi2 \setminus \gamma_n$ as $p$. Therefore, for $n$ large enough, the point
$c(t_n)$ is very close to $p$ and the angle $\alpha(t)$ is very close to
$\alpha_0$. This would imply, for $n$ large enough, that $p$ belongs to the
convex component of $\hi2 \setminus H_n$ and this would give again a
contradiction with the maximum principle for curves.
Therefore we get that $\alpha(t)\to 0$ as $t\to +\infty$.

\smallskip

To conclude the argument, we choose for $\hi2$  the model of the
unit disk equipped with the metric
$g_{\di}=\lambda^2 (z)\,\rvert dz \rvert^2$, where
$\lambda (z)=2/(1-\rvert z\rvert^2)$. We can assume that $p=0$ and that $C$ is
parametrized by arclength.

In polar coordinates we have $c(t)=(r(t)\cos \theta(t), r(t)\sin \theta(t))$
where $r(t)=\rvert c(t)\rvert >0$ and $\theta (t)\in \R$. We set
$\partial_r := (\cos \theta,\sin \theta)$.
Since $\alpha(t)\to 0$, we have
\begin{equation*}
 \frac{\langle c^\prime (t)\, ;\, \partial_r\rangle_\di}
{\rvert\partial_r\rvert_\di} \to -1,
\end{equation*}
where the scalar product and the norm are considered with
respect to the metric $g_\di$. From which we get that
$\lambda (c(t)) r^\prime (t) \to -1$.  Using that $c(t)\to p=0$ for $t\to\infty,$  we
obtain $ \lambda (c(t))\to 2$ for $t\to\infty$, therefore $r^\prime (t) \to -1/2$ and then
$r(t)\to -\infty$ as $t\to +\infty$. This is a contradiction and  this  concludes the
proof of the claim.

\

Since $C$ {   has infinite length} and  has no limit point in $\hi2$, we
deduce that its
asymptotic boundary
is not empty. Let $p_\infty \in \pain C$ be an asymptotic point of $C$.

\

 {\em {\bf Claim 2.}  $p_\infty$ is the unique asymptotic point of
 $C$.}

\smallskip

Let $\Gamma \subset \hi2$ be a complete curve with constant curvature $k$ such
that $p_\infty \not\in \pain \Gamma$ and $p_\infty$ belongs to the asymptotic
boundary of the convex component of $\hi2 \setminus  \Gamma$.

Let $\gamma \subset \hi2$ be a geodesic   intersecting $\Gamma$ such that
$p_\infty \in \pain \gamma$.
For any $s\in \R$, let  $\Gamma_s$ be the
translated copy of $\Gamma$ along $\gamma$ at distance $\rvert s \rvert$,
towards $p_\infty$ if $s>0$, and in the opposite direction otherwise.
We denote by $\Gamma_s^+$ the convex component of $\hi2 \setminus \Gamma_s$.
Thus, we have  $\Gamma_0=\Gamma$ and
$p_\infty \in \pain \Gamma_s^+$ for any $s \geq 0$.  Observe that $\cap_{s \geq 0} \pain \Gamma_s^+ =\{ p_\infty\}$.

  If we assume that   for any $s >0$, there exists
$t_s >0$ such that $c([t_s, +\infty\mathclose[) \subset \ov{\Gamma_s^+},$ then
 we deduce that $p_\infty$ is the unique asymptotic
point of $C,$ as desired.
Therefore,  we are left with the proof of  {    the assumption above}.

\smallskip

Suppose by contradiction that there exists $r >0$ such that for any $t >0$ the
curve
$c([t, +\infty\mathclose[)$ is not entirely contained in  $\ov{\Gamma_{r}^+}$.
Therefore, there is an arc $C_1\subset C$ such that
$\partial C_1 \subset \Gamma_r$ and $C_1 \cap  \Gamma_{r}^+ =\emptyset$, that
is $C_1$ stays outside  $\Gamma_{r}^+$. Note that $\ov C_1$ is a compact arc
with boundary on $\Gamma_r$.
Considering the curves $\Gamma_s$, for
$s$ going from $r$ to $-\infty$, we get a real number  $\rho < r$ such that
$C_1 \subset \ov{\Gamma_{\rho}^+}$ and $C_1$ and $ \Gamma_{\rho}$ are tangent
at some interior point of $C_1$. This gives a contradiction  by the
maximum principle,  keeping in mind the hypothesis about the
curvature of $C$
 and $\Gamma_\rho$ and the fact that $C_1$ belongs to the closure of the convex
component of $\hi2 \setminus \Gamma_\rho$. This concludes the proof .
\end{proof}

\begin{definition}\label{D.C1}
\begin{enumerate}
 \item \label{item.C0} Let $(p_n)$ be a sequence in $\hi2$ converging to some
point $p\in \hi2$.
 Let $v\in T_p\hi2$ and
$v_n \in T_{p_n} \hi2$ be non zero vectors. Assuming that
$p_n\not=p$, we
denote by $c_n$ the geodesic   passing through $p$ and $p_n$, and by $T_n$
the translation along $c_n$ such that
$T_n (p)=p_n$. If $p_n=p$, we set $T_n=Id$. Let $\alpha_n\in [0,\pi]$
be the non-oriented angle between $v_n$ and $T_n (v)$.

We say that the sequence $(v_n)$ {\em converges} to $v$
(denoted shortly by $v_n \rightarrow v$) if
$\alpha_n \rightarrow 0$ and
$|v_n|_{\hi2} \to |v|_{\hi2}$.

\item  \label{item.C1} Let $\gamma$ be a geodesic  in $\hi2$ and
$(\gamma_n)$ be a sequence
of  complete and regular $C^1$-curves in $\hi2.$ We say that  the
sequence $(\gamma_n)$ {\em converges   to $\gamma$ in the $C^1$ topology } if:
\begin{enumerate}
\item \label{C0} For any $\varepsilon >0$, there is $n_0\in \n$ such that for
any $n>n_0$ the curve $\gamma_n$ stays in the region of $\hi2$ bounded by the
two equidistant lines of $\gamma$ with distance $\varepsilon$ from $\gamma$.

\item \label{C1} Let $p\in \gamma$, and let $(p_n)$, $p_n\in
\gamma_n$,
be any sequence
converging to $p$. Let
$v_n\in T_{p_n}\gamma_n$ be a unit tangent vector of $\gamma_n$ at $p_n$.
If the sequence $v_n$ converges to a unit vector $v\in T_p\hi2$,
then $v\in T_p \gamma$, that is $v$ is tangent to
$\gamma$ at
$p$.
\end{enumerate}
\end{enumerate}
\end{definition}

\begin{prop}\label{P.C1}
 Let $\gamma  \subset \hi2$ be a geodesic  . Let $(\gamma_n)$ be a sequence
of  complete and regular $C^2$-curves such that:
\begin{itemize}
\item$\pain \gamma_n = \pain \gamma$ for any $n \in \n$.

\item$\sup _{q\in \gamma_n}\big\{\rvert \kappa_{\gamma_n}(q)
\rvert \big\}
\xrightarrow[n\to \infty]{} 0 $, where   $ \kappa_{\gamma_n}(q) $ is the
geodesic curvature of $\gamma_n$ at point
$q\in \gamma_n$.
\end{itemize}
 Then, the sequence $(\gamma_n)$ converges  to $\gamma$ in the $C^1$ topology.
 \end{prop}

\begin{proof}

 Let us prove (\ref{C0}) of Definition \ref{D.C1}.

Let $\varepsilon >0$ and let $n_0\in \n$ such that for any $n > n_0$ we have
\begin{equation*}
\tanh \varepsilon > \sup _{q\in \gamma_n}\big\{\rvert \kappa_{\gamma_n}(q)
\rvert
\big\}.
\end{equation*}
Denote by $L_\varepsilon ^1$ and $L_\varepsilon ^2$ the
two equidistant lines of $\gamma$ with distance  $\varepsilon$
from $\gamma$. It suffices to show that,
for any $n>n_0$, the curve $\gamma_n$ belongs to the
convex component of both $\hi2 \setminus L_\varepsilon ^1$ and
$\hi2 \setminus L_\varepsilon ^2$. We prove that fact for
$L_\varepsilon :=L_\varepsilon ^1$. The proof for $L_\varepsilon ^2$ is
analogous.

\

Let $\gamma^\prime$ be any geodesic of $\hi2$ different from $\gamma$,
intersecting $\gamma$. Let
$p_\infty ^\prime\in \pain \hi2 $ be the asymptotic point of $ \gamma^\prime$
which stays in the asymptotic boundary of the non convex component of
$\hi2\setminus L_\varepsilon$. Consider the translations along
$ \gamma^\prime,$ towards $p_\infty ^\prime$.
 Assume by contradiction that for any $n_0$ there exists $n>n_0$ such that
$\gamma_n$ does not belong to the convex component of
$\hi2 \setminus L_\varepsilon$. There exists a translated copy
$L_\varepsilon^\prime $ of $L_\varepsilon$ such that:

\begin{itemize}
 \item $L_\varepsilon^\prime$ intersects the curve $\gamma_n$  at one point
$q_n$.
\item $L_\varepsilon^\prime$ and  $\gamma_n$ are
tangent at $q_n$.

\item  $\gamma_n$ belongs to the closure of the convex
component of $\hi2\setminus L_\varepsilon^\prime$.
\end{itemize}
Since the geodesic curvature of $L_\varepsilon^\prime$ is
$\tanh \varepsilon$ (with
respect  to
the normal direction pointing towards the convex component of $\hi2\setminus L_\varepsilon^\prime$) and since
$\tanh \varepsilon >\sup _{q\in \gamma_n}\big\{\rvert \kappa_{\gamma_n}(q)
\rvert
\big\}$, we obtain
a contradiction with the maximum
principle. This completes the proof of (\ref{C0}).

\medskip

Now, we  prove (\ref{C1}) of Definition \ref{D.C1}.

By contradiction, assume that the unit vector $v\in T_p \hi2$ is not tangent to
$\gamma$.

Let $\varepsilon >0$ and let $L\subset \hi2$ be one of the two complete
curves  passing through $p$, tangent to $v$  whose absolute value of the
geodesic curvature is $\tanh \varepsilon$. Since $v$ is not tangent to
$\gamma$,
if $\varepsilon$ is small enough, then the curve $L$ separates the two
points of the asymptotic boundary of $\gamma$,
say $p_\infty$ and $q_\infty$.

As in the Definition \ref{D.C1}, we denote by $c_n$ the geodesic   passing
through
$p$ and $p_n$. Let $T_n$ be the hyperbolic translation along $c_n$ such
that $T_n(p)=p_n$. Let $R_n$ be the rotation in $\hi2$ around $p_n$ such that
$R_n ( T_n(v))=v_n$. Therefore, $L_n :=(R_n \circ T_n)(L)$ is a
complete curve through $p_n$, tangent to $\gamma_n$ at $p_n$,
with constant (absolute
value) curvature equal to $\tanh \varepsilon$.  If $n$ is large enough then the
curve
$L_n$ separates $p_\infty$ and $q_\infty$, since this is true for
$L$.

Observe that, if  $n$ is large enough, we have
$\sup _{q\in \gamma_n}\big\{\rvert \kappa_{\gamma_n}(q) \rvert \big\}
<\tanh \varepsilon$.
Consequently, using the maximum principle for curves in the same way as
before, we can show
that
$\gamma_n$ entirely belongs to the closure of one of the two components of
$\hi2 \setminus L_n$. But this gives a contradiction with the
assumption that
  $\pain \gamma_n = \pain \gamma=\{p_\infty, q_\infty\}$.
We conclude that $v$ is tangent to $\gamma$, as desired.
\end{proof}

\begin{remark}
 We can extend Definition \ref{D.C1} to any dimensional hyperbolic space
$\hi n,\ n\geq 2$. Moreover, we can prove in the same way as in Proposition
\ref{P.C1},
that if $\Pi\subset \hi n$ is a geodesic hyperplane and if $(\Pi_n)$ is a
sequence of
complete and regular $C^2$-hypersurfaces of $\hi n$ such that $\pain \Pi_n=\pain
\Pi$ for
any $n$ and $\sup_{q \in \Pi_n}\{| H_n (q)|\}\to 0 $, where $H_n(q)$ denotes the
mean
curvature of $\Pi_n$ at $q$,
then the sequence $(\Pi_n)$
converges
$C^1$ to $\Pi$.

\end{remark}

\begin{prop}\label{P.graph}
 Let $\gamma_1 \subset \hi2$ be a geodesic and let
$p_\infty\in \pain \hi2$ such that
$p_\infty\not\in  \pain\gamma_1$.

For any $\rho >0$, let $L_\rho$ be the equidistant
line to $\gamma_1$ whose  distance to $\gamma_1$ is $\rho$, such that
$p_\infty$ belongs to the asymptotic boundary of the
non convex component of $\hi2\setminus L_\rho$.

Let $0< k <1$ and let $c : [0,\infty\mathclose[ \rightarrow \hi2$
be a regular $C^2$-curve such that
$\pain c([0,\infty\mathclose[ )= \{p_\infty\}$ and such that
$\rvert \kappa (c(t))\rvert  <k$ for any $t\geq 0$.
Set $\rho_0=\max \{d_{\hi2}\big(c(0),\gamma_1\big), \tanh^{-1} (k)\}$.

 Then, for any $\rho > \rho_0$, one has the following facts.

 \begin{enumerate}
 \item \label{item.graph} The equidistant line $L_\rho$ cuts the curve
 $c([0,\infty\mathclose[ )$ at a unique point. Therefore, there exists $t_0>0$
such that the curve $c([t_0,\infty[)$ is a horizontal graph with respect to
$\gamma_1.$

\item \label{item.transversal} The equidistant line $L_\rho$ is
transversal to  $c([0,\infty\mathclose[ ).$
\end{enumerate}
\end{prop}

\begin{proof}
 Let $\rho >\rho_0$ and let $C= c\big([0,\infty\mathclose[ \big)$. Since
$\pain C=p_\infty$, the equidistant line $L_\rho$ must
intersect the curve $C$ at least at one point. Assume  by contradiction that
$L_\rho$ {   cut $C$ in at least two points}. By construction,
$c(0)$ belongs to the convex component of
$\hi2 \setminus L_\rho$. Let $p_1\in  L_\rho \cap C$ be
the first intersection  point from $c(0)$.
 The boundary maximum principle for curves
shows that the curves $ L_\rho $ and
$C$ are not tangent at $p_1$. Let
$p_2\in  L_\rho \cap C$ be the first intersection point
after $p_1$. Thus, the whole arc
 of $C$ between $p_1$ and $p_2$ belongs to the non convex component
of  $\hi2 \setminus L_\rho$. Now we obtain a contradiction with the maximum
principle in the following way.

\smallskip

 Let $p_\infty^\prime\in \pain \hi2$ be a point in the asymptotic boundary of
the convex component of $\hi2 \setminus  L_\rho$. Let
$\gamma\subset \hi2$ be the  geodesic   such that
$\pain \gamma =\{p_\infty, p_\infty^\prime\}$.
Considering the hyperbolic translations along
$\gamma$ towards $p_\infty$, we obtain a translated copy $L_\rho^\prime$ of
$L_\rho$ such that:

\begin{itemize}
 \item $L_\rho^\prime$ intersects the arc  of $C$  between $p_1$ and $p_2$   at one point $q$.
\item $L_\rho^\prime$ and  the arc of $C$  between $p_1$ and $p_2$ are tangent at $q$.
\item   The arc  of $C$  between $p_1$ and $p_2$ belongs to the closure of the convex
component of $\hi2\setminus L_\rho^\prime$.
\end{itemize}

Since the geodesic curvature of $L_\rho^\prime$ is $\tanh \rho$ (with respect to
the normal direction pointing  towards the  convex component of $\hi2\setminus L_\rho^\prime$) and since
$\tanh \rho > k > \sup_{q\in C}\{\rvert \kappa (C) \rvert\} $, we obtain
a contradiction with the maximum
principle.
So  Assertion (\ref{item.graph})  is proved.

\medskip

Now we prove Assertion (\ref{item.transversal}).

Suppose, by contradiction, that for some  $\rho > \rho_0$, the equidistant line $L_{\rho}$ is
tangent
to the curve $C$ at some point $p_1$.
{   Recall that the curvature of $L_\rho$ is strictly greater, in absolute
value, than the curvature of $C$. We deduce from the maximum principle that
an open arc of $C$, containing $p_1$, remains in the non convex component  of
$\hi2 \setminus L_\rho$.}

We set
$C_1=c\big(\mathopen] \rho_0, \infty\mathclose[\big)$. Then, the first part
of the proof shows that the curve $C_1\setminus \{p_1\}$ remains in the
non convex component of $\hi2 \setminus L_{\rho}$.
Therefore, for $\varepsilon >0$ small
enough, the equidistant line $L_{\rho + \varepsilon }$ intersects the curve
$C$
at least at two different points near $p_1$, giving a contradiction with
 assertion (1).
\end{proof}

\end{document}